\documentclass[12pt]{amsart}

\usepackage{amssymb,amsmath,amsfonts,amsthm,mathrsfs}
\usepackage{fullpage}
\usepackage{thumbpdf}
\usepackage{ifthen}
\usepackage{ifpdf}
\ifpdf
    \usepackage[ pdftex, plainpages = false, pdfpagelabels, 
                 pdfpagelayout = useoutlines,
                 bookmarks,
                 bookmarksopen = true,
                 bookmarksnumbered = true,
                 breaklinks = true,
                 linktocpage,
                 pagebackref,
                 colorlinks = false,  % was true
                 linkcolor = blue,
                 urlcolor  = blue,
                 citecolor = red,
                 anchorcolor = green,
                 hyperindex = true,
                 hyperfigures,
                 pdfauthor = Adam \ Topaz,
                 pdftitle =  Almost \ Commuting \ Pairs \ in \ Galois \ Groups
                 ]{hyperref} 
\else
    \usepackage[ dvips, 
                 bookmarks,
                 bookmarksopen = true,
                 bookmarksnumbered = true,
                 breaklinks = true,
                 linktocpage,
                 pagebackref,
                 colorlinks = true,
                 linkcolor = blue,
                 urlcolor  = blue,
                 citecolor = red,
                 anchorcolor = green,
                 hyperindex = true,
                 hyperfigures
                 ]{hyperref}
\fi
\usepackage[all]{xy}
\usepackage{color}
\usepackage[lite, backrefs, alphabetic]{amsrefs}

\DeclareMathOperator{\Div}{Div}
\DeclareMathOperator{\Hom}{Hom}
\DeclareMathOperator{\Gal}{Gal}

\DeclareMathOperator{\res}{res}
\DeclareMathOperator{\Spec}{Spec}

\DeclareMathOperator{\Char}{char}
\DeclareMathOperator{\td}{td}
\DeclareMathOperator{\tor}{tor}

\newtheorem{thm}{Theorem}

\newtheorem{prop}{Proposition}[section]
\newtheorem{lem}[prop]{Lemma}
\newtheorem{fact}[prop]{Fact}
\newtheorem{cor}[prop]{Corollary}

\theoremstyle{definition}
\newtheorem{defn}[prop]{Definition}
\newtheorem{example}[prop]{Example}
\newtheorem{remark}[prop]{Remark}

\numberwithin{equation}{section}

\newcommand{\F}{\mathbb{F}}
\newcommand{\Z}{\mathbb{Z}}
\newcommand{\Pbb}{\mathbb{P}}
\newcommand{\Abb}{\mathbb{A}}

\newcommand{\Q}{\mathbb{Q}}
\newcommand{\C}{\mathbb{C}}
\newcommand{\Gc}{\mathcal{G}}
\newcommand{\Hc}{\mathcal{H}}
\newcommand{\Ac}{\mathcal{A}}

\newcommand{\Oc}{\mathcal{O}}
\newcommand{\Cc}{\mathcal{C}}
\newcommand{\Sc}{\mathcal{S}}
\newcommand{\Vc}{\mathcal{V}}
\newcommand{\Tc}{\mathcal{T}}

\newcommand{\mf}{\mathfrak{m}}

\newcommand{\Uf}{\mathfrak{U}}
\newcommand{\Uc}{\mathcal{U}}
\newcommand{\Ib}{\mathbf{I}}
\newcommand{\Hb}{\mathbf{H}}

\newcommand{\poclaim}{\noindent\emph{Proof of Claim: }}

\newcommand{\Kl}{K^{\times \ell}}

\begin{document}

\title{Almost-Commuting-Liftable Subgroups of Galois Groups}
\author{Adam Topaz}
\address{Department of Mathematics, University of Pennsylvania, 209 S. 33rd Street, Philadelphia, PA 19104}
\date{Feb. 27th 2012}
\email{atopaz@math.upenn.edu}
\urladdr{www.math.upenn.edu/$\sim$atopaz}
\subjclass[2010]{12E30, 12F10, 12G05, 12J25}
\keywords{local theory, valuations, pro-$\ell$ Galois theory, Galois cohomology, abelian-by-central}

\begin{abstract}
Let $K$ be a field and $\ell$ be a prime such that $\Char K \neq \ell$.
In the presence of sufficiently many roots of unity in $K$, we show how to recover some of the inertia/decomposition structure of valuations inside the maximal $(\Z/\ell)$-abelian resp. pro-$\ell$-abelian Galois group of $K$ using its $(\Z/\ell)$-central resp. pro-$\ell$-central extensions.
\end{abstract}

\maketitle
%\tableofcontents

\section{Introduction}

The first key step in most strategies towards anabelian geometry is to develop a {\bf local theory}, by which one recovers decomposition groups of ``points'' using the given Galois theoretic information.
In the context of anabelian curves, one should eventually detect decomposition groups of closed points of the given curve within its \'etale fundamental group.
On the other hand, in the birational setting, this corresponds to detecting decomposition groups of arithmetically and/or geometrically meaningful places of the function field under discussion within its absolute Galois group. 
The first instance of such a local theory is Neukirch's group-theoretical characterization of decomposition groups of finite places of global fields;\footnote{This actually predates Grothendieck's anabelian geometry.} indeed, this was the basis for the celebrated Neukirch-Uchida theorem \cite{Neukirch1969}, \cite{Neukirch1969a}, \cite{Uchida1976}.
The Neukirch-Uchida theorem was expanded by Pop to all higher dimensional finitely generated fields by developing a local theory based on his $q$-Lemma \cite{Pop1994}.
The $q$-Lemma deals with the {\bf absolute pro-$q$ Galois theory}\footnote{Namely, fields whose absolute Galois group is a pro-$q$-group.} of fields and, as with Neukirch's result, works only in arithmetical situations.

On the other hand, at about the same time, two non-arithmetically based methods were proposed which recover inertia and decomposition groups of valuations from the {\bf relative pro-$\ell$ Galois theory} ($\ell$ a fixed prime) of a field whose characteristic is different from $\ell$.
The first relies on the theory of {\bf rigid elements} which was developed by several authors (see the details below).
This theory requires only that the field under discussion has characteristic different from $\ell$ and that it contains $\mu_\ell$; the input, however, must be the {\bf maximal pro-$\ell$ Galois group} of the field (cf. Engler-Koenigsmann \cite{Engler1998}).
Nevertheless, this method eventually led to the characterization of solvable absolute Galois groups of fields by Koenigsmann \cite{Koenigsmann2001}, and also the characterization of maximal pro-$\ell$ Galois groups of small rank \cite{Koenigsmann1998}, \cite{Efrat1998}.

The second method is Bogomolov and Tschinkel's theory of {\bf commuting-liftable pairs in Galois groups} \cite{Bogomolov2007}.\footnote{This theory was originally proposed by Bogomolov \cite{Bogomolov1991}.}
Its input is the much simpler {\bf maximal pro-$\ell$ abelian-by-central Galois group},\footnote{Terminology by Pop \cite{Pop2010}.} but it requires that the base field contain an {\bf algebraically closed subfield}.
Nevertheless, this theory was a key technical tool in the local theory needed to settle Bogomolov's program in birational anabelian geometry for function fields over the algebraic closure of finite fields -- see Bogomolov-Tschinkel \cite{Bogomolov2008a} in dimension $2$ and Pop \cite{Pop2011} in general. 

It goes without saying, the experts believed that the two approaches -- that of rigid elements versus that of commuting-liftable pairs -- were unrelated.
However, Pop suggested in his Oberwolfach report \cite{Pop2006a} that the two methods should be linked, even in the analogous $(\Z/\ell^n)$-abelian-by-central situation, but unfortunately never followed up with the details.
Also, the work done by Mah\'e, Min\'a\v{c} and Smith \cite{Mah'e2004} in the $(\Z/2)$-abelian-by-central situation, and Efrat-Min\'a\v{c} \cite{Efrat2011a} in special cases of the $(\Z/\ell)$-abelian-by-central situation suggest a connection between the two methods in this analogous context.

This paper aims to give a unifying approach for the two methods.
At the same time, we provide simpler arguments for the pro-$\ell$ abelian-by-central assertions and prove more general versions of these assertions which assume only that the field contains $\mu_{\ell^\infty}$.\footnote{Compare with \cite{Bogomolov2007} where the existence of an algebraically closed subfield is essential in the proof.}
For readers' sake, we give a short overview of the results mentioned above and see how the results of this paper fit into the larger context.

\subsection{Overview}

Let $K$ be a field with $\Char K \neq \ell$ which contains the $\ell^{{\rm th}}$ roots of unity $\mu_\ell \subset K$.
Denote by $K(\ell)$ the maximal pro-$\ell$ Galois extension of $K$ (inside a chosen separable closure of $K$) so that $\Gc_K := \Gal(K(\ell)|K)$ is the maximal pro-$\ell$ quotient of the absolute Galois group $G_K$ of $K$.
Let $w$ be a valuation of $K(\ell)$ and denote by $v = w|_K$ its restriction to $K$; denote by $k(w)$ the residue field of $w$ and $k(v)$ the residue field of $v$.
We denote by $I_{w|v} \leq D_{w|v} \leq \Gc_K$ the inertia resp. decomposition subgroup of $w|v$ inside $\Gc_K$.
Recall that $D_{w|v} / I_{w|v} = \Gc_{k(v)}$ and that the canonical short exact sequence 
\[ 1 \rightarrow I_{w|v} \rightarrow D_{w|v} \rightarrow \Gc_{k(v)} \rightarrow 1 \]
is split.
Moreover, if $\Char k(v) \neq \ell$, then $I_{w|v}$ is a free abelian pro-$\ell$ group of the same rank as $v(K^\times)/\ell$, and the action of $\Gc_{k(v)}$ on $I_{w|v}$ factors via the $\ell$-adic cyclotomic character.
Thus, if $\Char k(v) \neq \ell$, and $\sigma \in I_{w|v}$, $\tau \in D_{w|v}$ are given non-torsion elements so that the closed subgroup $\langle \sigma,\tau \rangle$\footnote{We denote by $\langle S \rangle$ the closed subgroup generated by $S$} is non-pro-cyclic, then ${\langle \sigma,\tau \rangle} = {\langle \sigma \rangle} \rtimes {\langle \tau \rangle} \cong \Z_\ell \rtimes \Z_\ell$ is a semi-direct product.

In a few words, the theory of rigid elements in the context of pro-$\ell$ Galois groups \cite{Engler1998} asserts that the only way the situation above can arise is from valuation theory, as described above.
More precisely, let $K$ be a field such that $\Char K \neq \ell$ and $\mu_{\ell} \subset K$.
If $\sigma,\tau \in \Gc_K$ are non-torsion elements such that ${\langle \sigma,\tau \rangle} = {\langle \sigma \rangle} \rtimes {\langle \tau \rangle}$ is non-pro-cyclic, then there exists a valuation $w$ of $K(\ell)$ such that, denoting $v = w|_K$, one has $\Char k(v) \neq \ell$, $v(K^\times) \neq v(K^{\times \ell})$, $\sigma \in I_{w|v}$ and $\tau \in D_{w|v}$.
The key idea in this situation is the explicit ``creation'' of valuation rings inside $K$ using rigid elements and so-called ``$\ell$-rigid calculus'' developed in \cite{Koenigsmann1995}.
Indeed, under the assumption that $\Gc_K = {\langle \sigma,\tau \rangle} = {\langle \sigma \rangle} \rtimes {\langle \tau \rangle}$ as above, \cite{Engler1998} shows that $K$ has sufficiently many ``strongly-rigid elements'' to produce an $\ell$-Henselian valuation $v$ of $K$ with $v(K^\times) \neq v(K^{\times \ell})$ and $\Char k(v) \neq \ell$, using \cite{Koenigsmann1995}.
Rigid elements were first considered by Ware \cite{Ware1981}, then further developed in the context of valuation theory by Arason-Elman-Jacob in \cite{Arason1987}, Koenigsmann in \cite{Koenigsmann1995}, Engler-Koenigsmann in \cite{Engler1998}, Efrat in \cite{Efrat1999}, \cite{Efrat2007} and also by others.

Assume, on the other hand, that $\mu_{\ell^\infty} \subset K$.
In this case, we denote by 
\[ \Pi_K^a := \frac{\Gc_K}{[\Gc_K,\Gc_K]}, \ \text{ and } \ \Pi_K^c := \frac{\Gc_K}{[\Gc_K,[\Gc_K,\Gc_K]]} \]
the maximal pro-$\ell$ abelian resp. maximal pro-$\ell$ abelian-by-central Galois groups of $K$ -- this terminology was introduced by Pop \cite{Pop2010}.
In the above context, assume again that $\Char k(v) \neq \ell$, then the $\ell$-adic cyclotomic character of $K$ (and of $k(v)$) is trivial. 
Hence, $\Gc_{k(v)}$ acts trivially on $I_{w|v}$; we conclude that $D_{w|v} \cong I_{w|v} \times \Gc_{k(v)}$ and recall that $I_{w|v}$ is abelian.
Denote by $K^{ab}$ the Galois extension of $K$ such that $\Gal(K^{ab}|K) = \Pi_K^a$, $v^{ab} := w|_{K^{ab}}$, $\widehat D_v := D_{v^{ab}|v}$ and $\widehat I_v := I_{v^{ab}|v}$; since $\Pi_K^a$ is abelian, $\widehat D_v$ and $\widehat I_v$ are independent of choice of $w$.
We deduce that for all $\sigma \in \widehat I_v$ and $\tau \in \widehat D_v$, there exist lifts $\tilde\sigma,\tilde\tau \in \Pi_K^c$ of $\sigma,\tau \in \Pi_K^a$ which commute in $\Pi_K^c$; since $\Pi_K^c$ is a central extension of $\Pi_K^a$, we conclude that {\bf any lifts} $\tilde\sigma,\tilde\tau \in \Pi_K^c$ of $\sigma,\tau \in \Pi_K^a$ commute as well -- such a pair $\sigma,\tau \in \Pi_K^a$ is called {\bf commuting-liftable}.

Bogomolov and Tschinkel's theory of commuting-liftable pairs \cite{Bogomolov2007} asserts that, under the assumption that $K$ contains an algebraically closed subfield $k = \bar k$, the only way a commuting pair can arise is via a valuation as described above.\footnote{It turns out that $\Char k(v) \neq \ell$ is not needed in order to produce a commuting-liftable pair, under a modified notion of decomposition and inertia. It turns out that valuations with residue characteristic equal to $\ell$ can and do arise from commuting-liftable pairs, as we will see in this paper.}
The method of loc.cit. uses the notion of a ``flag function;'' in particular, this is a homomorphism $K^\times \rightarrow \Z_\ell$ which corresponds, via Kummer theory, to an element in $\widehat I_v$ for some valuation $v$.\footnote{We generalize the notion of a flag function in this paper -- see, e.g., Remark \ref{rem:flag-funcs}.}
One then considers $\sigma,\tau$ as elements of $\Hom(K^\times,\Z_\ell) = \Hom(K^\times/k^\times,\Z_\ell)$ via Kummer theory, and produces the corresponding map:
\[ \Psi = (\sigma,\tau) : K^\times/k^\times \rightarrow \Z_\ell^2 \subset \Abb^2(\Q_\ell). \]
When one views $K^\times/k^\times = \Pbb_k(K)$ as an infinite dimensional projective space over $k$, the assumption that $\sigma,\tau$ are commuting liftable ensures that $\Psi$ sends {\bf projective lines} to {\bf affine lines}.
This severe restriction on $\Psi$ is then used to show that some $\Z_\ell$-linear combination of $\sigma$ and $\tau$ is a flag function.

As mentioned above, the theory of commuting-liftable pairs was originally outlined by Bogomolov in \cite{Bogomolov1991}, where he also introduced a program in birational anabelian geometry for fields of purely geometric nature -- i.e. function fields over an algebraically closed field of characteristic different from $\ell$ and dimension $\geq 2$ -- which aims to reconstruct such function fields $K$ from the Galois group $\Pi_K^c$.
If $\Char K > 0$, the above technical theorem eventually allows one to detect the decomposition and inertia subgroups of {\bf quasi-divisorial valuations} inside $\Pi_K^a$ using the group-theoretical structure encoded in $\Pi_K^c$ (see Pop \cite{Pop2010}).
In particular, for function fields $K$ over the algebraic closure of a finite field, one can detect the decomposition/inertia structure of {\bf divisorial valuations} inside $\Pi_K^a$ using $\Pi_K^c$.
While Bogomolov's program in its full generality is far from being complete, it has been carried through for function fields $K|k$, $k = \overline \F_p$ over the algebraic closure of a finite field -- by Bogomolov-Tschinkel \cite{Bogomolov2008a} in dimension 2, and by Pop \cite{Pop2011} in general.

In this paper, we obtain analogous results to those in the theory of commuting-liftable pairs, for both the $(\Z/\ell)$-abelian-by-central and the pro-$\ell$-abelian-by-central situations, by elaborating on and using the theory of rigid elements, while working under less restrictive assumptions than Bogomolov and Tschinkel's approach.
In particular, we reprove and generalize the main results of \cite{Bogomolov2007} using this method.
We begin by introducing some technical assumptions and notation.

\subsection{Notation}
For the remainder of this section, $\ell$ will denote a fixed prime, and $K$ a field whose characteristic is different from $\ell$.
A ``subgroup'' in the context of profinite groups will always mean a closed subgroup.

Assume that $\mu_\ell \subset K$.
In this case, we denote by $\Gc_K^a$ the maximal $(\Z/\ell)$-abelian and $\Gc_K^c$ the maximal $(\Z/\ell)$-abelian-by-central Galois groups of $K$.
Explicitly, denote by $\Gc_K^{(2)} := [\Gc_K,\Gc_K] \cdot (\Gc_K)^\ell$ and $\Gc_K^{(3)} = [\Gc_K,\Gc_K^{(2)}] \cdot (\Gc_K^{(2)})^\ell$, then 
\[ \Gc_K^a := \Gc_K/\Gc_K^{(2)}, \ \text{ and } \ \Gc_K^c := \Gc_K/\Gc_K^{(3)}. \]
The canonical projection $\pi : \Gc_K^c \twoheadrightarrow \Gc_K^a$ induces the following maps.
First, $[\bullet,\bullet] : \Gc_K^a \times \Gc_K^a \rightarrow \ker \pi$ defined by $[\sigma,\tau] = \tilde\sigma^{-1}\tilde\tau^{-1}\tilde\sigma\tilde\tau$ where $\tilde\sigma,\tilde\tau \in \Gc_K^c$ are some lifts of $\sigma,\tau \in \Gc_K^a$; since $\pi$ is a central extension, this is well-defined but moreover it is $(\Z/\ell)$-bilinear.
Second, $(\bullet)^\ell : \Gc_K^a \rightarrow \ker \pi$ defined by $\sigma^\ell = \tilde\sigma^\ell$ where, again, $\tilde\sigma \in \Gc_K^c$ is some lift of $\sigma \in \Gc_K^a$, and $\tilde\sigma^\ell$ denotes the $\ell^{\rm th}$-power of $\tilde\sigma$ in $\Gc_K^c$; since $\pi$ is a central extension with kernel killed by $\ell$, this map is well defined and, if $\ell \neq 2$, this map is $(\Z/\ell)$-linear.
To simplify the exposition, we denote by $\sigma^\beta = \sigma^\ell$ if $\ell \neq 2$ and $\sigma^\beta = 0 \in \ker \pi$ if $\ell = 2$ -- thus $(\bullet)^\beta$ is a linear map $\Gc_K^a \rightarrow \ker \pi$ regardless of $\ell$.
For a (closed) subgroup $A \leq \Gc_K^a$, we denote by
\[ \Ib(A) = \{ \sigma \in A \ : \ \forall \tau \in A, \ [\sigma,\tau] \in A^\beta \}. \]
Then $\Ib(A)$ is a subgroup of $A$ -- this is the so-called ``almost-commuting-liftable-center'' of $A$.

Suppose $v$ is a valuation of $K$.
We denote by $K'$ the Galois extension of $K$ such that $\Gal(K'|K) = \Gc_K^a$, and pick a prolongation $v'$ of $v$ to $K'$.
We denote by $D_v := D_{v'|v}$ and $I_{v} = I_{v'|v}$ the decomposition and inertia subgroups of $v'|v$ inside $\Gc_K^a$; since $\Gc_K^a$ is abelian, these groups are independent of choice of $v'$.
Moreover, we introduce the {\bf minimized} decomposition and inertia subgroups: 
\[ D_v^1 := \Gal(K'|K(\sqrt[\ell]{1+\mf_v})), \ \text{ and } \ I_v^1 := \Gal(K'|K(\sqrt[\ell]{\Oc_v^\times})). \]
Observe that $I_v^1 \leq D_v^1$; more importantly, however, $I_v^1 \leq I_v$ and $D_v^1 \leq D_v$ with equality whenever $\Char k(v) \neq \ell$ (see Proposition \ref{prop:decomp-thy-prop}).

On the other hand, suppose that $\mu_{\ell^\infty} \subset K$.
The canonical projection $\widehat \pi : \Pi_K^c \rightarrow \Pi_K^a$ induces a map
$[\bullet,\bullet] : \Pi_K^a \times \Pi_K^a \rightarrow \ker \widehat \pi$ defined by $[\sigma,\tau] = \tilde\sigma^{-1}\tilde\tau^{-1}\tilde\sigma\tilde\tau$ for some lifts $\tilde\sigma,\tilde\tau \in \Pi_K^c$ of $\sigma,\tau \in \Pi_K^a$; similarly to above, this map is well-defined and $\Z_\ell$-linear. 
For a closed subgroup $\widehat A \leq \Pi_K^a$, we denote by 
\[ \Ib(\widehat A) = \{ \sigma \in \widehat A \ : \ \forall \tau \in \widehat A, \ [\sigma,\tau] = 0 \}. \]
Again, $\Ib(\widehat A)$ is a subgroup of $\widehat A$ and can be seen as the ``commuting-liftable-center'' of $\widehat A$.

Suppose again that $v$ is a valuation of $K$.
We denote by $K^{ab}$ the Galois extension of $K$ such that $\Gal(K^{ab}|K) = \Pi_K^a$, and pick a prolongation $v^{ab}$ of $v$ to $K^{ab}$.
As above, we denote by $\widehat D_v := D_{v^{ab}|v}$ and $\widehat I_{v} = I_{v^{ab}|v}$ the decomposition and inertia subgroups of $\Pi_K^a$ associated to $v^{ab}|v$; since $\Pi_K^a$ is abelian, these groups are independent of choice of $v^{ab}$.
Moreover, we introduce the {\bf minimized} decomposition and inertia subgroups in this context:
\[ \widehat D_v^1 := \Gal(K^{ab}|K(\sqrt[\ell^\infty]{1+\mf_v})), \ \text{ and } \ \widehat I_v^1 := \Gal(K^{ab}|K(\sqrt[\ell^\infty]{\Oc_v^\times})). \]
As in the $\Z/\ell$-case, one has $\widehat I_v^1 \leq \widehat D_v^1$; more importantly, $\widehat I_v^1 \leq \widehat I_v$ and $\widehat D_v^1 \leq \widehat D_v$ with equality whenever $\Char k(v) \neq \ell$ (see Proposition \ref{prop:decomp-thy-prop}).

\subsection{Main Results of the Paper}
\label{sec:main-results-manuscr}

We are now ready to summarize the main results of this paper.
The following theorem summarizes Theorems \ref{thm:acl-groups} and \ref{thm:main-acl}.

\begin{thm}
Let $K$ be a field such that $\Char K \neq \ell$ and $\mu_{2\ell} \subset K$.
Let $Z \leq \Gc_K^a$ be given and consider $\Ib(Z) \leq Z$ as defined above.
Then the following hold:
\begin{itemize}
\item If $\Ib(Z) = Z$, there exists a valuation $v$ of $K$ and a subgroup $B \leq Z$ such that $Z/B$ is cyclic, $Z \leq D_v^1 \leq D_v$ and $B \leq I_v^1 \leq I_v$.
\item If $\Ib(Z) \neq Z$, there exists a valuation $v$ of $K$ such that $Z \leq D_v^1 \leq  D_v$ and $\Ib(Z) = Z \cap I_v^1 \leq Z \cap I_v$.
\end{itemize}

Assume moreover that $\mu_{\ell^\infty} \subset K$.
Let $\widehat Z \leq \Pi_K^a$ be given and assume that $\Pi_K^a / \widehat Z$ is torsion-free.
Consider $\Ib(\widehat Z) \leq \widehat Z$ as defined above.
Then the following hold:
\begin{itemize}
\item If $\Ib(\widehat Z) = \widehat Z$, there exists a valuation $v$ of $K$ and a subgroup $\widehat B \leq \widehat Z$ such that $\widehat Z/\widehat B$ is pro-$\ell$-cyclic, $\widehat Z \leq \widehat D_v^1 \leq \widehat D_v$ and $\widehat B \leq \widehat I_v^1 \leq \widehat I_v$.
\item If $\Ib(\widehat Z) \neq \widehat Z$, there exists a valuation $v$ of $K$ such that $\widehat Z \leq \widehat D_v^1 \leq  \widehat D_v$ and $\Ib(\widehat Z) = \widehat Z \cap \widehat I_v^1 \leq \widehat Z \cap \widehat I_v$.
\end{itemize}
\end{thm}

We also give a condition which ensures that also $\Char k(v) \neq \ell$. This is summarized in the following theorem which follows from Theorems \ref{thm:acl-groups-with-char} and \ref{thm:main-acl-with-char}.
\begin{thm}
Let $K$ be a field such that $\Char K \neq \ell$ and $\mu_{2\ell} \subset K$.
Let $Z \leq \Gc_K^a$ be given and consider $\Ib(Z) \leq Z$ as defined above.
Denote by $L = (K')^Z$ and assume that there exists $Z_0 \leq \Gc_L^a$ such that $\Ib(Z_0) \leq Z_0$ map surjectively onto $\Ib(Z) \leq Z$ via the canonical map $\Gc_L^a \rightarrow \Gc_K^a$.
Then the following hold:
\begin{itemize}
\item If $\Ib(Z) = Z$, there exists a valuation $v$ of $K$ with $\Char k(v) \neq \ell$ and a subgroup $B \leq Z$ such that $Z/B$ is cyclic, $Z \leq  D_v$ and $B \leq I_v$.
\item If $\Ib(Z) \neq Z$, there exists a valuation $v$ of $K$ with $\Char k(v) \neq \ell$ such that $Z \leq D_v$ and $\Ib(Z) = Z \cap I_v$.
\end{itemize}

Assume furthermore that $\mu_{\ell^\infty} \subset K$.
Let $\widehat Z \leq \Pi_K^a$ be given such that $\Pi_K^a/\widehat Z$ is torsion-free and consider $\Ib(\widehat Z) \leq \widehat Z$ as defined above.
Denote by $L = (K^{ab})^{\widehat Z}$ and assume that there exists $\widehat Z_0 \leq \Pi_L^a$ such that $\Pi_L^a/\widehat Z_0$ is torsion-free and $\Ib(\widehat Z_0) \leq \widehat Z_0$ map surjectively onto $\Ib(\widehat Z) \leq \widehat Z$ via the canonical map $\Pi_L^a \rightarrow \Pi_K^a$.
Then the following hold:
\begin{itemize}
\item If $\Ib(\widehat Z) =\widehat Z$, there exists a valuation $v$ of $K$ with $\Char k(v) \neq \ell$ and a subgroup $\widehat B \leq \widehat Z$ such that $\widehat Z/\widehat B$ is cyclic, $\widehat Z \leq  \widehat D_v$ and $\widehat B \leq \widehat I_v$.
\item If $\Ib(\widehat Z) \neq \widehat Z$, there exists a valuation $v$ of $K$ with $\Char k(v) \neq \ell$ such that $\widehat Z \leq \widehat D_v$ and $\Ib(\widehat Z) = \widehat Z \cap \widehat I_v$.
\end{itemize}
\end{thm}

We also give a group-theoretical recipe to detect precisely $I_v^1 \leq D_v^1 \leq \Gc_K^a$ resp. $\widehat I_v^1 \leq \widehat D_v^1 \leq \Pi_K^a$ -- see Theorem \ref{thm:main-ACL-compatibility}.
Moreover, we provide several applications of this theory.
First, we give a group theoretical recipe to detect $I_v \leq D_v \leq \Gc_K^a$ resp. $\widehat I_v \leq \widehat D_v \leq \Pi_K^a$ for valuations $v$ of $K$ with $\Char k(v) \neq \ell$, using a somewhat larger group than $\Gc_K^c$ resp. $\Pi_K^c$, but one that is still much smaller than $\Gc_K$ -- see \S~\ref{sec:detect-inert-groups}.
We also prove a sufficient condition which detects whether or not $\Char K = 0$ using the group-theoretical structure encoded in a characteristic quotient of $\Gc_K$ -- see \S~\ref{sec:maximal-pro-ell} which also includes another application towards the possible structure of $\Gc_K$.

\subsection{A Guide Through the Paper}
\label{sec:guide-thro-manuscr}

The paper has two paths: the {\bf mod $\ell$ case} which deals with the groups $K^\times/\ell$, $\Gc_K^c$, $\Gc_K^a$ and which at some points makes the assumption that $\mu_\ell \subset K$ or $\mu_{2\ell} \subset K$, and the {\bf pro-$\ell$ case} which deals with the groups $\widehat K$ (the $\ell$-adic completion of $K^\times)$, $\Pi_K^c$, $\Pi_K^a$ and which at some points makes the assumption that $\mu_{\ell^\infty} \subset K$.
While the two can be considered completely separately, the arguments are sufficiently similar to merit only a single paper.
The statements of all lemmas, propositions, theorems etc. in this paper have first the mod-$\ell$ case and then the pro-$\ell$ case.
The arguments are usually only given for the pro-$\ell$ case; usually the pro-$\ell$ case is more technical and it is essentially a simple matter of making the appropriate changes in notation to deduce the analogous mod-$\ell$ arguments.
The exception to this is in \S \ref{sec:gener-cohom-pro-ell} where we give most of the arguments in the mod-$\ell$ case. 
Indeed the existence of a possibly non-trivial Bockstein map in the mod-$\ell$ case makes this case more technical; as in the other sections, the corresponding arguments in the pro-$\ell$ case are essentially identical to those in the mod-$\ell$ case.
Also, Theorem \ref{thm:rigid-to-valuative} is explicitly proved in both cases -- see the argument itself for a remark towards a unified proof.

The main body of the paper is split up into two sections.
\S~\ref{sec:rigid-elements} deals mostly with fields $K$ that have no further restrictions.
\begin{itemize}
\item In \S~\ref{sec:rigid-valu-theory} we recall the notion of a rigid subgroup $T \leq K^\times$ and show how such subgroups relate to valuations of $K$.
\item In \S~\ref{sec:rigidity-milnor-k} we explore the relationship between rigidity of subgroups $T \leq K^\times$ and the structure of the Milnor K-ring $K_*^M(K)$.
We also explore analogous relationships in the $\ell$-adically complete case.
In \S~\ref{sec:ACL-groups}, this will eventually allow us to describe rigid subgroups in terms of Galois theory via Kummer theory and Galois cohomology.
\item In \S~\ref{sec:main-results-rigid} we state the main results of the section which allow us to detect valuations $v$ of $K$ using the structure of $K_*^M(K)$ resp. its $\ell$-adically complete analogue.
One should note that these results impose no restrictions on the field $K$.
\item In \S~\ref{sec:restr-char} we describe further conditions that ensure $\Char k(v) \neq \ell$ as soon as $\Char K \neq \ell$.
\item In \S~\ref{sec:compatibility-coars} we explore the compatibility of the results of \S~\ref{sec:main-results-rigid} in taking coarsenings/refinements of $v$.
In this subsection, we also describe explicitly the ``minimal'' valuations which can be detected with this method, and show that divisorial valuations satisfy this minimality condition (see Example \ref{example:div-vals}).
\end{itemize}

\S~\ref{sec:ACL-groups} deals exclusively with fields $K$ such that $\Char K \neq \ell$ and $\mu_\ell \subset K$ resp. $\mu_{\ell^\infty} \subset K$.
We will in some cases make the assumption that further $\mu_{2\ell} \subset K$ which ensures, in particular, that $-1 \in \Kl$; note that this is a vacuous condition if $\ell \neq 2$.
\begin{itemize}
\item In \S~\ref{sec:hilb-decomp-theory} we review some basic facts from Hilbert's decomposition theory of valued fields $(K,v)$ such that $\Char K \neq \ell$, $\Char k(v) \neq \ell$ and $\mu_\ell \subset K$ resp. $\mu_{\ell^\infty} \subset K$ -- this expands on the brief overview from the introduction.
We will be able to explicitly describe the Kummer-duals of the decomposition/inertia subgroups of such valuations $v$ of $K$ inside $\Gc_K^a$ resp. $\Pi_K^a$.
\item In \S~\ref{sec:gener-cohom-pro-ell} we review some basic facts from the cohomology of pro-$\ell$ groups which will be immediately used in \S~\ref{sec:galois-cohomology}.
\item In \S~\ref{sec:galois-cohomology} we recall some facts from the Galois cohomology of fields $K$ such that $\Char K \neq \ell$ and $\mu_\ell \subset K$ resp. $\mu_{\ell^\infty} \subset K$.
This allows us to describe a condition for $A \leq \Gc_K^a$ resp. $\widehat A \leq \Pi_K^a$ which ensures that the Kummer dual of $A$ resp. $\widehat A$ is rigid.
\item In \S~\ref{sec:main-results-acl} we describe the main results mentioned in the introduction.
These are merely a translation of the main results of \S~\ref{sec:rigid-elements}, via Kummer theory, to group theoretical results within $\Gc_K^c$ resp. $\Pi_K^c$ along with the canonical projection $\Gc_K^c \twoheadrightarrow \Gc_K^a$ resp. $\Pi_K^c \rightarrow \Pi_K^a$.
\end{itemize}

As mentioned above, in \S~\ref{sec:applications} we present some applications of the main results of \S~\ref{sec:ACL-groups}.
In \S~\ref{sec:detect-inert-groups}, we apply our results to describe group-theoretical recipes which detect inertia/decomposition groups in $\Gc_K^a$ resp. $\Pi_K^a$ of valuations $v$ of $K$ whose residue characteristic is different from $\ell$.
As a separate application, in \S~\ref{sec:maximal-pro-ell} we prove a result which restricts the group theoretical structure of the pro-$\ell$ group $\Gc_K$ -- see Corollary \ref{cor:application-char}.
This corollary can be seen as a sufficient condition to detect whether a field $K$ has characteristic $0$ using only the group theoretical data encoded in a characteristic quotient of $\Gc_K$ which we denote by $\Gc_K^m$ (see \S~\ref{sec:detect-inert-groups} for the definition of $\Gc_K^m$).
Moreover, in Corollary \ref{cor:pro-ell-comm-pairs}, we describe minimal conditions which describe elements which commute in $\Gc_K$ given they commute in its small quotients.

\subsection*{Acknowledgments}
I would like to thank all who expressed interest in this work and in particular Florian Pop, Jakob Stix, Jochen Koenigsmann, Moshe Jarden, Dan Haran and Lior Bary-Soroker.

\section{Rigid Elements}
\label{sec:rigid-elements}

Throughout we use the term ``valuation'' for a Krull valuation.
If $K$ is a field and $v$ is a valuation of $K$, we denote by $\Oc_v$ the valuation ring, $\mf_v$ the valuation ideal, $\Gamma_v$ the value group and $k(v) = \Oc_v/\mf_v$ the residue field of $v$.

Throughout, we will work with a fixed prime $\ell$.
For any field $K$, we denote by $\widehat K$ the $\ell$-adic completion of $K^\times$; i.e. $\widehat K = \lim_n K^\times/\ell^n$.
In general, for an abelian group $W$ we denote by $\widehat W = \lim_n W/\ell^n$ the $\ell$-adic completion of $W$.
Also, throughout we will also use the notation $\widehat \bullet$ to differentiate the pro-$\ell$ case from the mod-$\ell$ case.
We consider $\widehat K$ as a complete $\Z_\ell$-module and we implicitly only consider submodules which are complete and closed; we generally use the notation $\widehat T \leq \widehat K$ for such a submodule.
Each $\widehat T \leq \widehat K$ is given by a compatible system of subgroups $(T_n \leq K^\times/\ell^n)_n$ such that $T_{n+1} \mod K^{\times \ell^n} = T_n$.

For $\widehat T \leq \widehat K$ given by $(T_n)_n$ as above, we use the notation $K^\times \cap \widehat T \leq K^\times$ to denote $\bigcap_n T_n$ where we consider $T_n$ as a subgroup of $K^\times$ which contains $K^{\times \ell^n}$.
Alternatively, $K^\times \cap \widehat T$ is the inverse image of $\widehat T$ under the canonical $\ell$-adic completion map $K^\times \rightarrow \widehat K$.
For an element $x \in K^\times$ we use the notation $x \in \widehat T$ to mean $x \in K^\times \cap \widehat T$, and similarly for subsets $S \subset K^\times$.

\subsection{Rigidity and Valuation Theory}
\label{sec:rigid-valu-theory}

Let $(K,v)$ be a valued field.
The ultrametric inequality ensures that for all $x \in K^\times$ such that $x \notin \Oc_v^\times$, one has $1+x \in \Oc_v^\times \cup x \cdot \Oc_v^\times$.
The same is true if we replace $\Oc_v^\times$ with any subgroup $H$ such that $\Oc_v^\times \leq H \leq K^\times$ and, as we outline in the next lemma, this condition is almost sufficient for the existence of $v$.

\begin{lem}
\label{lem:valuative}
Let $K$ be a field and let $H \leq K^\times$ be given.
The following are equivalent:
\begin{enumerate}
\item There exists a valuation $v$ of $K$ such that $\Oc_v^\times \leq H$.
\item $-1 \in H$, for all $x \in K^\times \smallsetminus H$ one has $1+x \in H \cup x H$, and whenever $x,y \in K^\times \smallsetminus H$ are such that $1+x,1+y \in H$, one has $1+x(1+y) \in H$.
\end{enumerate}
\end{lem}
\begin{proof}
First assume that there exists a valuation $v$ such that $\Oc_v^\times \leq H$.
Let $x \in K^\times \smallsetminus H$ be given.
Then, in particular, $v(x) \neq 0$ and thus $1+x \in \Oc_v^\times$ iff $v(x) > 0$; also, $1+x \in x \cdot \Oc_v^\times$ iff $v(x) < 0$. 
Thus $1+x \in H \cup x \cdot H$ for all such $x$.
Moreover, if $x,y \notin H$ and $1+x,1+y \in H$ one has $v(x),v(y) > 0$ and thus $v(x\cdot(1+y)) > 0$ so that $1+x(1+y) \in H$ as required.

The converse is \cite{Arason1987} Theorem 2.10 taking $T=H$ in loc.cit..
Indeed, we note that the our assumption ensures the ``preadditive'' condition of loc.cit..
\end{proof}

\begin{defn}
\label{defn:valuative}
Let $K$ be a field and $H \leq K^\times$.
We will say that $H$ is {\bf valuative} if it satisfies the equivalent statements of Lemma \ref{lem:valuative}.
Similarly, if $\widehat H \leq \widehat K$, we say that $\widehat H$ is valuative provided that $K^\times \cap \widehat H$ is valuative.
\end{defn}

\begin{remark}
\label{rem:flag-funcs}
Let $f : K^\times \rightarrow \Z/\ell$ be a surjective homomorphism.
Then the kernel of $f$ is valuative if and only if $f$ is a {\bf flag function} in the sense of Bogomolov-Tschinkel \cite{Bogomolov2007}.
Similarly, the kernel of a surjective homomorphism $\widehat f : \widehat K \rightarrow \Z_\ell$ is valuative if and only if $f$ is a flag function in the sense of loc.cit.. 
\end{remark}

A subgroup $H \leq K^\times$ is valuative if and only if there exists some valuation $v$ such that $K^\times \twoheadrightarrow K^\times/H$ factors through $v : K^\times \rightarrow \Gamma_v$.
Similarly, $\widehat H \leq \widehat K$ is valuative if and only if the canonical map $\widehat K \twoheadrightarrow \widehat K / \widehat H$ factors through $\widehat v : \widehat K \rightarrow \widehat \Gamma_v$ for some valuation $v$.
Recall that any $\widehat H \leq \widehat K$ is given by a compatible system of $H_n \leq K^\times/\ell^n$ such that $H_{n+1} \mod \ell^n = H_n$.
If we consider each $H_n$ as a subgroup of $K^\times$ which contains $K^{\times \ell^n}$ we note that $\widehat K \rightarrow \widehat K / \widehat H$ factors through $\widehat v$ if and only if $\Oc_v^\times \leq H_n$ for all $n$ and thus $\Oc_v^\times \leq \bigcap_n H_n = K^\times \cap \widehat H$.

To each valuative subgroup $H \leq K^\times$, we will associate a canonical valuation $v_H$ as follows.
This valuation is obtained by first picking any valuation $w$ of $K$ such that $\Oc_w^\times \leq H$ and then taking the coarsest coarsening of $w$, say $w_H$, such that $\Oc_{w_H}^\times \leq H$.
As we show in the next lemma, $w_H$ does not depend on the original choice of $w$ and we denote this unique valuation by $v_H$.
For a valuative $\widehat H \leq \widehat K$, we denote by $v_{\widehat H}$ the valuation $v_H$ for $H = K^\times \cap \widehat H$.

In the particular case where $\Kl \leq H \leq K^\times$, $\Char K \neq \ell$ and $\mu_\ell \subset K$, the valuation $v_H$ is precisely the usual $K(\sqrt[\ell]{H})$-core of $w$ where $\Oc_w^\times \leq H$.
Also in the case where $\mu_{\ell^\infty} \subset K$ and $\widehat H \leq \widehat K$ is valuative, $v_{\widehat H}$ is the usual $K\left(\sqrt[\ell^\infty]{K^\times \cap \widehat H}\right)$-core of $w$.
See, for instance, \cite{Pop2010} and/or \cite{Pop2011a} which follow classical results of F. K. Schmidt.

\begin{lem}
\label{lem:v_H}
Let $K$ be a field and let $H$ be a valuative subgroup of $K^\times$.
Then there exists a unique coarsest valuation $v_H$ such that $\Oc_{v_H}^\times \leq H$.
If $w$ is a valuation of $K$ such that $\Oc_w^\times \leq H$, then $v_H$ is a coarsening of $w$; moreover $w = v_H$ if and only if $w(H)$ contains no non-trivial convex subgroups.
\end{lem}
\begin{proof}
Let $w$ be any valuation such that $\Oc_w^\times \leq H$ and consider the coarsening $v$ of $w$ which corresponds to the quotient of $\Gamma_w$ by the maximal convex subgroup of $w(H)$.
This is the coarsest coarsening $v$ of $w$ such that $\Oc_v^\times \leq H$.
By construction, $v(H)$ contains no non-trivial convex subgroups.
We deduce that whenever $x,y \in K^\times$ such that $v(x) = v(y) \mod v(H)$ but $v(x) < v(y)$, there exists a $z \in K^\times$ such that $v(x),v(y) \neq v(z) \mod v(H)$ and $v(x) < v(z) < v(y)$.

Now suppose $h \in H$ and $x \notin H$.
Then $v(h) \neq v(x)$; moreover $v(h) < v(x)$ iff $h+x \in H$ and $v(h) > v(x)$ iff $h+x \in x \cdot H$.
An element $h \in H$ such that $1+x = h+x \mod H$ for all $x \in K^\times \smallsetminus H$ must be in $\Oc_v^\times$ by the discussion above.
We deduce that $\Oc_v^\times$ depends only on $H$ and $K$, but not at all on the original choice of $w$.
Indeed, $\Oc_v^\times$ is precisely the set of all $h \in H$ such that for all $x \in K^\times \smallsetminus H$ one has $1+x = h+x \mod H$.
\end{proof}

\begin{remark}
\label{rem:vsubH-vs-arasons-valuation}
Suppose $H$ is valuative and denote by $v = v_H$.
Then $\Oc_v$ is precisely the valuation ring $\Oc(H,H)$ constructed in \cite{Arason1987}.
\end{remark}

\begin{defn}
\label{defn:v_H}
Given a valuative subgroup $H \leq K^\times$, we denote by $v_H$ the canonical valuation associated to $H$ as per Lemma \ref{lem:v_H}.
Namely, $v_H$ is the unique coarsest valuation of $K$ such that $\Oc_{v_H}^\times \leq H$. 
For $\widehat H \leq \widehat K$, we denote by $v_{\widehat H}$ the canonical valuation associated to $K^\times \cap \widehat H$ as above.
Namely, $v_H$ is the unique coarsest valuation such that $K^\times \rightarrow K^\times/H$ factors through $v_H$ and $v_{\widehat H}$ is the coarsest valuation such that $\widehat K \rightarrow \widehat K/\widehat H$ factors through $\widehat v_{\widehat H}$.
If furthermore $K^{\times \ell} \leq H$, then $\Gamma_{v_H}$ contains no non-trivial $\ell$-divisible convex subgroups.
Similarly, $\Gamma_{v_{\widehat H}}$ contains no non-trivial $\ell$-divisible convex subgroups.
\end{defn}

Our last technical lemma involving valuative subgroups gives a condition which ensures two valuations are comparable.
%This lemma resembles a technical step in the proof of \cite{Pop2010} Proposition 2.5(1).
%However, we note that the lemma below imposes no restrictions on the characteristic of $K$ or the residue fields of the valuations, nor on the subgroup $H \lneq K^\times$.
We note that the lemma below imposes no restrictions on the characteristic of $K$ or the residue fields of the valuations, nor on the subgroup $H \lneq K^\times$.

\begin{lem}
\label{lem:non-valuative-comparable}
Let $K$ be a field.
Suppose that $H \lneq K^\times$ is a {\bf proper} subgroup such that $H$ is {\bf not valuative}.
Let $v_1,v_2$ be two valuations such that $1+\mf_{v_1},1+\mf_{v_2} \leq H$.
Then $v_1$ and $v_2$ are comparable.

In particular, suppose $\widehat H \lneq \widehat K$ is proper and non-valuative.
Suppose $v_1,v_2$ are two valuations such that $1+\mf_{v_1}, 1+\mf_{v_2} \leq \widehat H$.
Then $v_1$ and $v_2$ are comparable.
\end{lem}
\begin{proof}
For simplicity, we denote $(\Oc_{v_i},\mf_{v_i}) = (\Oc_i,\mf_i)$.
As $(1+\mf_1) \cdot (1+\mf_2) \subset H \neq K^\times$, the two valuations $v_1$ and $v_2$ are dependent by the approximation theorem.
Consider the non-trivial valuation $v$ such that $\Oc_v = \Oc_{v_1} \cdot \Oc_{v_2}$ and denote by $w_i = v_i/v$ the corresponding valuations on $k(v)$.
As $v$ is a coarsening of $v_i$, we also have $1+\mf_v \subset H$.
Moreover, if both $w_1$ and $w_2$ are non-trivial, then they must be independent.
Denote by $H_v$ the image in $k(v)^\times$ of $H \cap \Oc_v^\times$.
Observe that $(\Oc_v^\times \cdot H) / H$ is canonically isomorphic to $k(v)^\times/H_v$, and since $\Oc_v^\times \not\subset H$, we deduce that $H_v \neq k(v)^\times$.
However, $1+\mf_{v_i} \subset H$ implies that $1+\mf_{w_i} \subset H_v$ and thus $(1+\mf_{w_1}) \cdot (1+\mf_{w_2}) \subset H_v$.
In particular, either $w_1$ or $w_2$ must be trivial for otherwise they would be independent.
Thus $v_1$ and $v_2$ are comparable.
\end{proof}

We are now ready to introduce our notion of a {\bf rigid} subgroup $T \leq K^\times$ resp. $\widehat T \leq \widehat K$; our terminology is motivated by Koenigsmann \cite{Koenigsmann1995}.
As we will see in Theorem \ref{thm:rigid-to-valuative}, rigid subgroups are intimately tied to valuative subgroups and vice-versa.
In the following subsection, we explore how rigid subgroups behave in Milnor K-theory.
Moreover, in the next section we translate our results for rigid subgroups, using Kummer theory, to subgroups of the maximal $(\Z/\ell)$-abelian Galois group of fields $K$ with $\Char K \neq \ell$ and $\mu_{2\ell} \subset K$ resp. subgroups of the maximal pro-$\ell$ abelian Galois group of fields $K$ with $\Char K \neq \ell$ and $\mu_{\ell^\infty} \subset K$.
However, for the remainder of this section $K$ will denote an arbitrary field unless otherwise noted.

\begin{defn}
\label{defn:rigid}
Let $K$ be a field and $T \leq K^\times$ such that $ \Kl\leq T$.
We say that $T$ is {\bf rigid} if for all $x \in K^\times \smallsetminus T$ one has
\[ \langle x,1-x \rangle \mod T \ \text{ is cyclic in } K^\times/T. \] 
If $\widehat T \leq \widehat K$, we say that $\widehat T$ is rigid provided that for all $x \in K^\times$ such that $x \notin \widehat T$ one has 
\[ \langle x, 1-x \rangle \mod \widehat T \ \text{ is } \ \Z_\ell\text{-cyclic in } \widehat K/\widehat T. \]
\end{defn}

The following lemma outlines some basic properties of rigid and valuative subgroups.

\begin{lem}
\label{lem:rigid-valuative-properties}
Let $K$ be a field.
Then the following hold.
\begin{enumerate}
\item If $\Kl \leq T \leq K^\times$ is rigid (resp. valuative) and $T \leq H \leq K^\times$ then $H$ is rigid (resp. valuative).
\item If $\Kl \leq T_i \leq K^\times$ are given such that $T_i$ is rigid and $T_{ij} := T_i \cap T_j$ is rigid for all $i,j$, then $T = \bigcap_i T_i$ is rigid.
\item Suppose $\Kl \leq T \leq K^\times$ is rigid and that for each $i$: $H_i$ is valuative and $T \leq H_i \leq K^\times$. Then $H = \bigcap_i H_i$ is valuative.
\end{enumerate}

Similarly:
\begin{enumerate}
\item If $\widehat T \leq \widehat K$ is rigid (resp. valuative) and $\widehat T \leq \widehat H \leq \widehat K$ then $\widehat H$ is rigid (resp. valuative).
\item If $\widehat T_i \leq \widehat K$ are given such that $\widehat K/\widehat T_i$ is torsion-free, $\widehat T_i$ is rigid and $\widehat T_{ij} := \widehat T_i \cap \widehat T_j$ is rigid for all $i,j$, then $\widehat T = \bigcap_i \widehat T_i$ is rigid.
\item Suppose $\widehat T \leq \widehat K$ is rigid and that for each $i$: $\widehat H_i$ is valuative, $\widehat T \leq \widehat H_i \leq \widehat K$ and $\widehat K / \widehat H_i$ is torsion-free. Then $\widehat H = \bigcap_i \widehat H_i$ is valuative.
\end{enumerate}
\end{lem}
\begin{proof}
Below we show the pro-$\ell$ situation, as the mod-$\ell$ case is similar (and simpler).
(1) and (2) are an easy consequence of the definitions, and are left to the reader.
Let us show (3) by proving the second condition of Lemma \ref{lem:valuative} for $\widehat H$; alternatively, (3) can be deduced from Theorem \ref{thm:rigid-to-valuative} combined with Lemma \ref{lem:non-valuative-comparable}.
Assume WLOG that $\widehat H =  \widehat T$ and note that $-1 \in \widehat H$.
Suppose $x \in K^\times \smallsetminus \widehat H$.
Then there exists an index $i$ such that $x \notin \widehat H_i$ and thus $1+x = 1 \mod \widehat H_i$ or $1+x = x \mod \widehat H_i$.
On the other hand, there exists $z \in \widehat K$ such that $(x \mod \widehat H), (1+x \mod \widehat H) \in \langle z \mod \widehat H \rangle$.
Say $z^a = 1+x \mod \widehat H$ and $z^b = x \mod \widehat H$.
As $x \notin \widehat H_i$ one has $z \notin \widehat H_i$. Moreover, one also has $1+x = (z^{bc}) \mod \widehat H_i$ for $c = 0$ or $c = 1$.
But then $z^a = z^{bc} \mod \widehat H_i$ and since $\widehat K/\widehat H_i$ is torsion-free, we deduce that $a = bc$.
This shows that whenever $x \in K^\times$, $x \notin \widehat H$, one has $1+x \in \widehat H \cup x \cdot \widehat H$; denoting by $H = K^\times \cap \widehat H$, one has $1+x \in H \cup x \cdot H$ for all $x \notin H$.

A similar argument as above also shows that whenever $x,y \in K^\times$, $x,y \notin \widehat H_i$ such that $1+x,1+y \in \widehat H_i$, one has $1+x(1+y) \in \widehat H$ and so $1+x(1+y) \in H$.
So now we assume that $x \in \widehat H_j \smallsetminus \widehat H_i$ but $y \in \widehat H_i \smallsetminus \widehat H_j$ where $i \neq j$, and $1+x,1+y \in \widehat H$.
Then $\langle x, y \rangle \mod \widehat H$ is non-cyclic.
Consider then $1+x(1+y) = 1+x \cdot h$ for some $h \in K^\times \cap \widehat H$ and so $1+x(1+y) = 1 \mod \widehat H$ or $1+x(1+y) = x \mod \widehat H$.
On the other hand, $1+x(1+y) = 1+x+xy = h' + xy$ for some $h' \in K^\times \cap \widehat H$.
Thus $1+x(1+y) = 1 \mod \widehat H$ or $1+x(1+y) = xy \mod \widehat H$.
But $x,y$ are $\Z_\ell$-independent in $\widehat K / \widehat H$ so that $1+x(1+y) \in \widehat H$; thus, $1+x(1+y) \in H$, as required. 
\end{proof}

\begin{remark}
\label{rem:minimal-valuative-subgroup-for-a-rigid-subgroup}
As an immediate consequence of Lemma \ref{lem:rigid-valuative-properties}(3), we deduce that, if $T \leq K^\times$ is given such that $K^{\times \ell} \leq T$ and $T$ is rigid, then there exists a {\bf unique minimal} $H$, such that $T \leq H \leq K^\times$ and $H$ is valuative.
Similarly, if $\widehat T \leq \widehat K$ is given such that $\widehat K / \widehat T$ is torsion-free and $\widehat T$ is rigid, then there exists a {\bf unique minimal} $\widehat H$ such that $\widehat T \leq \widehat H \leq \widehat K$, $\widehat K / \widehat H$ is torsion-free and $\widehat H$ is valuative.
\end{remark}

The mod-$\ell$ case of the following theorem is a straight-forward application of established results; in fact, one direction of the the mod-$\ell$ case can be explicitly found in \cite{Koenigsmann1998} Proposition 3.1.
Below, we include the mod-$\ell$ situation alongside our pro-$\ell$ case for sake of completeness.

\begin{thm}
\label{thm:rigid-to-valuative}
Let $K$ be an arbitrary field and let $T$ be given such that $\langle \Kl,-1 \rangle \leq T \leq K^\times$.
The following are equivalent:
\begin{enumerate}
\item There exists $H$ such that $T \leq H \leq K^\times$, $H/T$ is cyclic (possibly trivial), $H$ is valuative and $1+\mf_{v_H} \leq T$.
\item $T$ is rigid.
\end{enumerate}

Let $\widehat T \leq \widehat K$ be given such that $\widehat K/\widehat T$ is torsion-free.
Then the following are equivalent:
\begin{enumerate}
\item There exists $\widehat H$ such that $\widehat T \leq \widehat H \leq \widehat K$, $\widehat H/\widehat T$ is $\Z_\ell$-cyclic (possibly trivial), $\widehat H$ is valuative and $1+\mf_{v_{\widehat H}} \leq \widehat T$.
\item $\widehat T$ is rigid.
\end{enumerate}
\end{thm}
\begin{proof}
First we show the mod-$\ell$ case.
Assume (1) and denote by $v = v_H$.
Let $x \notin T$ be given.
If $v(x) > 0$ then $1-x \in T$, while $v(x) < 0$ implies that $(1/x) \cdot (1-x) \in T$.
On the other hand, if $v(x) = 0$ then $x \in H$ and so $H/T = \langle x \mod T \rangle$.
Now, $v(1-x) = 0$ for otherwise $1-x \in \mf_v$ and so $x \in T$.
Thus, $1-x \in H$ so that $\langle 1-x,x \rangle  \mod T$ is cyclic.

The converse is \cite{Koenigsmann1998} Lemma 3.3 combined with \cite{Arason1987} Theorem 2.16 as follows.
Denote by $H$ the subgroup of $K^\times$ generated by all $x \in K^\times$ such that $T + x \cdot T \not\subset T \cup x \cdot T$ (in particular, $T \leq H$).
By \cite{Koenigsmann1998} Lemma 3.3, $H/T$ is cyclic in the case where $\ell \neq 2$ while $H = T$ if $\ell = 2$.
By \cite{Arason1987} Theorem 2.16, there exists $\widetilde H$ such that $H \leq \widetilde H \leq K^\times$, $[\widetilde H : H] \leq 2$, $\widetilde H$ is valuative, and $1+\mf_{v_{\widetilde H}} \subset T$.
If $\ell \neq 2$ then $\widetilde H = H$; in any case then, $\widetilde H / T$ is cyclic, as required.

In the pro-$\ell$ case, the fact that $(1) \Rightarrow (2)$ is similar. 
Let us show that $(2) \Rightarrow (1)$.
Denote by $T = K^\times \cap \widehat T$ and note that $-1 \in T$.
Denote by $H$ the subgroup of $K^\times$ generated by $x \in K^\times$ such that $T+x \cdot T \not\subset T \cup x \cdot T$ (note that $T \leq H$), and we denote by $\widehat H$ the (complete) submodule of $\widehat K$ generated by $H$ and $\widehat T$.
We proceed to show that $\widehat H / \widehat T$ is $\Z_\ell$-cyclic.
Suppose not, then there exist $x,y \in K^\times$ such that $\langle x , y \rangle \mod \widehat T$ is non-cyclic, and both $1+x \notin T \cup x T$, $1+y \notin T \cup yT$.
On the other hand, $\widehat T$ is rigid.
Thus, $1+x \in \langle x \mod \widehat T \rangle \otimes_{\Z_\ell} \Q_\ell$ and $1+y \in \langle y \mod \widehat T \rangle \otimes_{\Z_\ell} \Q_\ell$; denote by $V := \langle x \mod \widehat T, y \mod \widehat T \rangle \otimes_{\Z_\ell} \Q_\ell$.
Moreover, denote by $V_0$ the collection of all $z \in \widehat K / \widehat T$ such that there exists $a \in \Z_\ell$ with $z^a \in \langle x , y \rangle \mod \widehat T$.
Then $V_0$ is a free $\Z_\ell$-module of rank two and the canonical map $V_0 \rightarrow V$ is injective.
We abuse the notation and denote by $z = z \mod \widehat T$ for $z \in \widehat K$ and/or $z \in K^\times$; moreover, we consider the basis $x,y$ for $V$ and we identify $V = \Q_\ell^2$ using this basis.
Finally, we embed $\Q_\ell^2$ into $\Pbb^2(\Q_\ell)$ via $(a,b) \mapsto (1:a:b)$.
Lastly, we abuse the notation even more by denoting the elements of $\Q_\ell^2$ by $(a,b) = (1:a:b)$ while elements of $\Pbb^2(\Q_\ell) \smallsetminus \Q_\ell^2$ by $(a:b) = (0:a:b)$.
For $v,w \in \Pbb^2(\Q_\ell)$ such that $v \neq w$, we denote by $l(v,w)$ the unique line between $v$ and $w$, considering $\Pbb^2(\Q_\ell)$ as a projective space.

Note, the rigidity property of $\widehat T$ implies the following: Let $x',y' \in K^\times \cap V_0$ be such that $x' \neq y'$ considered as elements of $V$ (equivalently, $x' \neq y'$ considered as elements of $V_0$).
Then $x'+y' \in l(x',y') \cap V_0$.
Indeed, $x'+y' = x'(1+y'/x') = x' \cdot (y'/x')^a = (x')^{1-a} \cdot (y')^{a}$ for some $a \in \Q_\ell$.

Our strategy is now motivated by \cite{Bogomolov2007} Proposition 4.1.2; however, we do not assume the existence of an algebraically closed subfield of $K$ as in loc.cit. and this makes our proof much more technical. We obtain an obvious contradiction in the following claim:

\noindent{\bf Claim: } In the situation above, consider the unique isomorphism $\Psi \in \text{PGL}(\Pbb^2(\Q_\ell))$ such that $\Psi(0,0) = (0,0)$, $\Psi(0,1) = (0,1)$, $\Psi(1,0) = (1,0)$, $\Psi(1+x) = (1:0)$ and $\Psi(1+y) = (0:1)$.
Then $\Psi^{-1}(\Pbb^2(\Q)) \subset V_0$.

This provides a contradiction as $\Pbb^2(\Q) \subset \Pbb^2(\Q_\ell)$ is dense in the $\ell$-adic topology and so the same is true for $\Psi^{-1}(\Pbb^2(\Q))$.
However, $V_0$ is a $\Z_\ell$-lattice $V_0 \subset \Q_\ell^2 = \Abb^2(\Q_\ell) \subset \Pbb^2(\Q_\ell)$ and is not dense -- contradiction.
Before we prove this claim, we show how to finish the proof of the theorem under the assumption of the claim.
We deduce from the above contradiction that $\widehat H/\widehat T$ is cyclic.
Now we proceed using \cite{Arason1987} Theorem 2.16 in a similar fashion to the mod-$\ell$ case.
Indeed, loc.cit. implies that there exists $\widehat H_1 \leq \widehat K$ such that $[\widehat H_1 : \widehat H] \leq 2$, $\widehat H_1$ is valuative, and $1+\mf_{v_{\widehat H_1}} \subset \widehat T$.
If $\ell \neq 2$, this completes the proof as $\widehat H_1 = \widehat H$.
If $\ell = 2$, we note that $\widehat H_1/\widehat T$ is still cyclic by our torsion-free assumption on $\widehat K/\widehat T$, as required.

\

\poclaim
Recall that $\Psi(x) = (1,0)$, $\Psi(y) = (0,1)$, $\Psi(1+x) = (1:0)$ and $\Psi(1+y) = (0:1)$.
Recall furthermore that the rigidity property ensures that for all $x',y' \in K^\times \cap V_0$, $\Psi(x'+y')$ is contained in the line between $\Psi(x')$ and $\Psi(y')$.
To complete the proof of the claim, we show inductively that $(\star)$: $a+bx+cy \in V_0$ and $\Psi(a+bx+cy) = (b+c-a:b:c)$ for coprime integers $a,b,c$.
When we say $(\star)$ is satisfied by $(b+c-a:b:c) \in \Pbb^2(\Q_\ell)$, $a,b,c$ coprime integers, we mean that indeed $a+bx+cy \in V_0$ and $\Psi(a+bx+cy) = (b+c-a:b:c)$.

The proof of the claim contains many steps and each step relies heavily on the previous ones.
The key idea is to write an element of $K^\times$ as a sum/difference (of elements of $V_0$) in two ways.
This will then force this element to be in $V_0$ and its image under $\Psi$ to be in the intersection of the corresponding lines.
For example, $1+x+y = (1+x) + y = (1+y) + x$ and thus $\Psi(1+x+y)$ lies in the intersection $l(\Psi(1+x),\Psi(y)) \cap l(\Psi(1+y),\Psi(x))$ where $l(v,w)$ denotes the line between $v$ and $w$ in $\Pbb^2(\Q_\ell)$; since $\Psi(1+x) = (1:0)$, $\Psi(x) = (1,0)$, $\Psi(1+y) = (0:1)$ and $\Psi(y) = (0,1)$, we deduce that $\Psi(1+x+y) = (1,1)$.
In the many steps that follow, we omit the explicit details as above and, when needed, give the two sums.
\begin{enumerate}
\item $\Psi(1+x+y) = (1,1)$ since $1+x+y = (1+x)+y = (1+y)+x$, as above.
\item $\Psi(2+x+y) = (1:1)$ since $2+x+y = 1+(1+x+y) = (1+x)+(1+y)$.
\item $\Psi(x-y) = (-1:1)$ since $x-y = (1+x) - (1+y)$.
\item $\Psi(2+2x+y) = (2,1)$ since $2+2x+y =x+(2+x+y) = (1+x)+(1+x+y)$.
\item $\Psi(1+2x) = (2,0)$ since $1+2x = (2+2x+y) - (1+y) = x+(1+x)$.
\item Let $m \in \Z$ be given, $m \geq 1$. If $\Psi((m-1)+mx) = (m,0)$ for $m \geq 1$ then $\Psi((m+1)+(m+1)x+y) = (m+1,1)$ since $(m+1)+(m+1)x+y = (2+x+y) + ((m-1)+mx) = (m+mx+y)+(1+x)$.
\item Let $m \in \Z$ be given, $m \geq 1$. If $\Psi((m+1)+(m+1)x+y) = (m+1,1)$ then $\Psi(m+(m+1)x) = (m+1,0)$ since $m+(m+1)x = ((m+1)+(m+1)x+y) - (1+y) = ((m-1)+mx)+(1+x)$.
\item Thus, $\Psi(m+(m+1)x) = (m+1,0)$ and $\Psi(m+mx+y) = (m,1)$. Similarly, $\Psi(m+(m+1)y) = (0,m+1)$ and $\Psi(m+x+my) = (1,m)$.
\item Since $(m+mx+y) + (2+x+y) = (m+2)+(m+1)x+2y = ((m+1)+(m+1)x+y)+(1+y)$ we deduce that $\Psi((m+2)+(m+1)x+2y) = (m+1,2)$. Similarly, $\Psi((m+2)+2x+(m+1)y) = (2,m+1)$ and in this way we deduce $(\star)$ for the integer lattice in the first quadrant by replacing $(0,0)$, $(1,0)$ and $(0,1)$ with $(m,m)$, $(m+1,m)$ and $(m,m+1)$ respectively ($m > 0$), and proceeding inductively.
\item One has $\Psi(2+x) = (-1,0)$ since $2+x = (2+x+y) - y = 1+(1+x)$ and similarly $\Psi(2+y) = (0,-1)$.
\item Arguing in a similar way to above using $\Psi(2+x) = (-1,0)$ and $\Psi(2+y) = (0,-1)$, we deduce $(\star)$ for the whole lattice $\Z^2$.
\item $\Psi(x+y) = (1/2,1/2)$ since $x+y = (1+x+y) - 1$ and therefore $\Psi(x-1) = (1/2,0)$ since $x-1 = (x+y) - (1+y) = x-1$; similarly, $\Psi(y-1) = (0,1/2)$.
\item $\Psi(3+x) = (-1/2,0)$ since $3+x = (2+x) + 1 = (1-y) + (2+x+y)$ and similarly $\Psi(3+y) = (0,-1/2)$. In this way one obtains $(\star)$ for the half-lattice $1/2 \cdot \Z^2$, using the same process as above.
\item $\Psi(-1+x+y) = (1/3,1/3)$ since $-1+x+y = (-1+x)+y = (-1+y)+x$ and thus $\Psi(-2+x) = (1/3,0)$ since $-2+x = (-1+x+y)-(1+y) = (-1+x) - 1$; similarly, $\Psi(-2+y) = (0,1/3)$.
\item For $m > 0$, $m \in \Z$, inductively we deduce that $\Psi(-m+x+y) = (1/(m+2),1/(m+2))$ since $-m + x + y = (-m+y)+x = (-m+x)+y$. $\Psi(-m+x) = (1/(m+1),0)$ since $-m+x = (-(m-1)+x+y)-(1+y) = (-(m-1)+x)-1$ and similarly $\Psi(-m+y) = (0,1/(m+1))$. In this way one obtains $(\star)$ for $\Q^2$ in the first quadrant. 
\item Similarly to above, one has $\Psi(m+1+x) = (-1/m,0)$ and $\Psi(m+1+y) = (0,-1/m)$ for all $m > 0$, $m \in \Z$ and we thereby obtain $(\star)$ for $\Q^2$ by arguing as above.
\item For simplicity, denote by $x_\infty = 1+x$ and $y_\infty = 1+y$. An easy inductive argument shows that $\Psi(x_\infty+n\cdot y_\infty) = (1:n)$ for all $n \in \Z$. Similarly, $\Psi(m\cdot x_\infty+y_\infty) = (m:1)$ for all $m \in \Z$.
\item We then deduce that $\Psi(m\cdot x_\infty+n\cdot y_\infty) \in l(x_\infty,y_\infty)$ when $\gcd(m,n) = 1$ since $m\cdot x_\infty+n\cdot y_\infty = (x_\infty+(n-1)\cdot y_\infty) + ((m-1) \cdot x_\infty + y_\infty)$.
\item Finally, $\Psi(m \cdot x_\infty + n \cdot y_\infty) = (m:n)$ for $m,n$ with $\gcd(m,n) = 1$ inductively since $(m+n-1+mx+ny) + 1 = m \cdot x_\infty + n \cdot y_\infty$.
\end{enumerate}
We've thus proven $(\star)$ and the claim.
One should note that the contradiction already occurs in Step (6) above in the case where $\Char K > 0$; the required contradiction already occurs in Step (16) for $K$ with $\Char K = 0$ since $\Q^2 = \Abb^2(\Q) \subset \Abb^2(\Q_\ell)$ is also dense in $\Pbb^2(\Q_\ell)$.
Moreover, note that a similar claim, replacing $\Z_\ell$ and $\Q_\ell$ with $\Z/\ell$, could have been used as an alternative to \cite{Koenigsmann1998} Lemma 3.3 in the mod-$\ell$ case.
In any case, this completes the proof of Theorem \ref{thm:rigid-to-valuative}.
\end{proof}

\subsection{Rigidity and Milnor K-theory}
\label{sec:rigidity-milnor-k}

Let $K$ be a field.
The usual construction of the Milnor K-ring is as follows:
\[ K_n^M(K) = \frac{(K^\times)^{\otimes n}}{\langle a_1 \otimes \cdots \otimes a_n \ : \ a_i + a_j = 1 \ \text{ for some } \ 1 \leq i < j \leq n \rangle}. \]
The tensor product makes $K_*^M(K) := \bigoplus_n K_n^M(K)$ into a graded-commutative ring and we denote by $\{\bullet,\bullet\}$ the product $K_1^M(K) \times K_1^M(K) \rightarrow K_2^M(K)$.

More generally, let $T \leq K^\times$ be given.
We define $K_*^M(K)/T$ as the quotient of $K_*^M(K)$ by the graded ideal generated by $T \leq K^\times = K_1^M(K)$ or explicitly as follows:
\[ K_n^M(K)/T = \frac{(K^\times/T)^{\otimes n}}{\langle a_1 \cdot T \otimes \cdots \otimes a_n \cdot T \ : \ 1 \in a_i \cdot T + a_j \cdot T \ \text{ for some } \ 1 \leq i < j \leq n \rangle}. \]
Again, the tensor product makes $K_*^M(K)/T$ into a graded-commutative ring and we denote the product in this ring by $\{\bullet,\bullet\}_T$.
Moreover, one has a surjective map of graded-commutative rings: $K_*^M(K) \twoheadrightarrow K_*^M(K)/T$.
It is well known that $\{-1,x\} = \{x,x\} \in  K_2^M(K)$, for all $x \in K^\times$.
Thus the same is true in $K_2^M(K)/T$; namely, $\{-1,x\}_T = \{x,x\}_T$.
For more on the arithmetic properties of these canonical quotients of the Milnor K-ring, refer to Efrat \cite{Efrat2006}, \cite{Efrat2007} where they are systematically studied.

We define $\widehat K_n^M(K)$ as the $\ell$-adic completion of $K_n^M(K)$ which makes $\widehat K_*^M(K) := \bigoplus_n \widehat K_n^M(K)$ into a graded $\Z_\ell$-algebra which is complete in each degree.
For $\widehat T \leq \widehat K = \widehat K_1^M(K)$, we denote by $\widehat K_*^M(K)/\widehat T$ as follows.
Say $\widehat T$ is given by the compatible system $(T_m)_m$, $K^{\ell^n} \leq T_m \leq K^\times$ as above.
We denote $\widehat K_n^M(K)/\widehat T = \lim_{m} K_n^M(K)/T_m$.
This makes $\widehat K_*^M(K)/\widehat T = \bigoplus_n \widehat K_n^M(K)/\widehat T$ into a graded $\Z_\ell$-algebra which is, again, complete in each degree.

\begin{defn}
\label{defn:wedge}
Let $K$ be a field and let $T \leq K^\times$ be given such that $\Kl \leq T$ and assume that $-1 \in T$.
We denote by $\wedge^2_T K$ the quotient of $(K^\times/T) \otimes_\Z (K^\times/T)$ by the subgroup generated by elements of the form $x \otimes x$.
In particular if $\ell \neq 2$ this is precisely the usual wedge product of the $\Z/\ell$ vector space $K^\times/T$ with itself.

Since $-1 \in T$ and $\{-1,x\}_T = \{x,x\}_T = 0$ we deduce that in any case one has a canonical surjective map:
\[ \wedge^2_T K \rightarrow K_2^M(K)/T \]
which is compatible with products of elements from $K^\times/T = K_1^M(K)/T$.

In a similar way, for $\widehat T \leq \widehat K$ such that $\widehat K/\widehat T$ is torsion-free, we denote by $\widehat \wedge^2_{\widehat T} K$ the $\ell$-adic completion of $(\widehat K/\widehat T)\otimes_{\Z_\ell}(\widehat K/\widehat T)/\langle x \otimes x, \ x \in \widehat K/\widehat T \rangle$.
Since $\widehat K/\widehat T$ is torsion-free, one has $-1 \in \widehat T$ and, again, one has a canonical surjective map:
\[ \widehat\wedge^2_{\widehat T} K \rightarrow \widehat K_2^M(K)/\widehat T. \]
\end{defn}

The rigidity condition of $T \leq K^\times$ resp. $\widehat T \leq \widehat K$ translates to a useful condition in the structure of $K_*^M(K)/T$ resp. $\widehat K_*^M(K)/\widehat T$:

\begin{lem}
\label{lem:rigid-iff-wedge}
Let $K$ be a field and let $T \leq K^\times$ be given such that $\Kl \leq T$ and assume that $-1 \in T$.
The following are equivalent:
\begin{enumerate}
\item $T$ is rigid.
\item The canonical map $\wedge^2_T K \rightarrow K_2^M(K)/T$ is an isomorphism.
\item For all $x,y \in K^\times/T$ such that $x,y$ are $(\Z/\ell)$-independent in $K^\times/T = K_1^M(K)/T$, one has $\{x,y\}_T \neq 0 \in K_2^M(K)/T$.
\end{enumerate}

Let $\widehat T \leq \widehat K$ be given such that $\widehat K / \widehat T$ is torsion-free.
The following are equivalent:
\begin{enumerate}
\item $\widehat T$ is rigid.
\item The canonical map $\widehat \wedge^2_{\widehat T} K \rightarrow \widehat K_2^M(K)/\widehat T$ is an isomorphism.
\item For all $x,y \in \widehat K/\widehat T$ such that $\langle x,y \rangle \mod \widehat T$ is a non-cyclic $\Z_\ell$-module in $\widehat K/\widehat T = \widehat K_1^M(K)/\widehat T$, one has $\{x,y\}_{\widehat T} \neq 0 \in \widehat K_2^M(K)/\widehat T$.
\end{enumerate}
\end{lem}
\begin{proof}
We show the pro-$\ell$ case here; the mod-$\ell$ case is similar.
Assume (1).
The kernel of the canonical surjective map $\widehat \wedge^2_{\widehat T} K \rightarrow \widehat K_2^M(K)/\widehat T$ is generated by elements of the form $(x \cdot \widehat T)\wedge ((1-x)\cdot \widehat T)$ as $x$ varies over all elements of $K^\times$ such that $x \notin \widehat T$.
However, all such elements are already trivial in $\widehat \wedge^2_{\widehat T} K$ as $\widehat T$ is rigid; thus we deduce (2).
$(3) \Rightarrow (1)$ is obvious since $\{1-x,x\} = 0$ for all $x \in K^\times$, $x \neq 1$.
$(2) \Rightarrow (3)$ follows immediately from the definition of $\wedge^2_T K$ resp. $\widehat \wedge^2_{\widehat T} K$.
\end{proof}

\begin{defn}
\label{defn:Z-of-T}
Let $K$ be a field and let $T \leq K^\times$ be given such that $\Kl \leq T$ and $-1 \in T$.
Denote by $\Hb(T)$ the intersection of all $H \leq K^\times$ such that:
\begin{itemize}
\item $T \leq H \leq K^\times$ and $K^\times / H$ is cyclic (and possibly trivial).
\item For all $H'$ such that $T \leq H' \leq K^\times$ and $K^\times/H'$ is cyclic, the intersection $H \cap H'$ is rigid.
\end{itemize}

Let $\widehat T \leq \widehat K$ be given such that $\widehat K / \widehat T$ is torsion-free.
We denote by $\Hb(\widehat T)$ the intersection of all $\widehat H \leq \widehat K$ such that:
\begin{itemize}
\item $\widehat T \leq \widehat H \leq \widehat K$ and $\widehat K / \widehat H$ is torsion-free cyclic (and possibly trivial).
\item For all $\widehat H'$ such that $\widehat T \leq \widehat H' \leq \widehat K$ and $\widehat K/\widehat H'$ is torsion-free cyclic, the intersection $\widehat H \cap \widehat H'$ is rigid.
\end{itemize}
\end{defn}

\begin{remark}
\label{rem:ZofT}
We first remark that $\Hb(T)$ resp. $\Hb(\widehat T)$ might be ``trivial'' (i.e. equal to $K^\times$ resp. $\widehat K$).
Also, $\Hb(T) = T$ if and only if $T$ is rigid ($\Kl \leq T \leq K^\times$) and $\Hb(\widehat T) = \widehat T$ if and only if $\widehat T$ is rigid ($\widehat T \leq \widehat K$ and $\widehat K / \widehat T$ is torsion-free).
Using Lemma \ref{lem:rigid-iff-wedge}, we deduce that $\Hb(T)$ resp. $\Hb(\widehat T)$ can be completely realized, as subgroups of $K^\times/T$ resp. $\widehat K / \widehat T$, using only the structure of $K_*^M(K)/T$ resp. $\widehat K_*^M(K)/\widehat T$.
On the other hand, the subgroups $\Hb(T)$ resp. $\Hb(\widehat T)$ have an alternative K-theoretic definition which we describe below.

Let $\Kl \leq T \leq K^\times$ be given.
From the definition of $\Hb(T)$ along with Lemma \ref{lem:rigid-valuative-properties}, we deduce that whenever $T \leq H \leq \Hb(T)$ is such that $\Hb(T) / H$ is cyclic, then $H$ is rigid.
Because of this, $\Hb(T)$ satisfies the following property: whenever $x \in K^\times \smallsetminus \Hb(T)$ and $y \in K^\times \smallsetminus T$ are such that $x,y$ are $(\Z/\ell)$-independent in $K^\times/T$, then $\{x,y\}_T \neq 0$.
Indeed, the images of $x,y$ in $K^\times/H$ are also independent for some $T \leq H \leq \Hb(T)$ such that $\Hb(T)/H$ is cyclic.
Moreover, arguing as in Lemma \ref{lem:rigid-iff-wedge}, $\Hb(T)$ is the unique minimal subgroup of $K^\times$ which contains $T$ and satisfies this property.

On the other hand, let $\widehat T \leq \widehat K$ be given such that $\widehat K / \widehat T$ is torsion-free.
It follows from the definition of $\Hb(\widehat T)$ that $\widehat K / \Hb(\widehat T)$ is torsion-free.
Using Lemma \ref{lem:rigid-valuative-properties} as above, we deduce that if $\widehat T \leq \widehat H \leq \Hb(\widehat T)$ is such that $\Hb(\widehat T) / \widehat H$ is torsion-free cyclic, then $\widehat H$ is rigid.
Similarly to above, we deduce that $\Hb(\widehat T)$ satisfies the following property: whenever $x \in \widehat K \smallsetminus \Hb(\widehat T)$ and $y \in \widehat K \smallsetminus \widehat T$ are such that $\langle x,y \rangle \mod \widehat T$ is non-cyclic in $\widehat K / \widehat T$, then $\{x,y\}_{\widehat T} \neq 0$.
Moreover, arguing as in Lemma \ref{lem:rigid-iff-wedge}, $\Hb(\widehat T)$ is the unique minimal submodule of $\widehat K$ which contains $\widehat T$ and satisfies this property among all submodules $\widehat H$ such that $\widehat K/\widehat H$ is torsion-free.
\end{remark}

\subsection{Main Results}
\label{sec:main-results-rigid}

We are now ready to present and prove the main results of this section which allow us to detect valuations from the Milnor K-theory of a field $K$.
One should note that a more restricted version of the mod-$\ell$ case in the propositions below may be deduced from \cite{Efrat2007}.\footnote{Loc.cit. requires that either $\ell = 2$ or, in our notation, $K^\times/T$ is finite.}

\begin{prop}
\label{prop:rigid-kthy}
Let $K$ be an arbitrary field.
Let $-1 \in T \leq K^\times$ be given and assume that $T = \Hb(T)$.
Then there exists an $H$ such that $T \leq H \leq K^\times$, $H/T$ is cyclic (possibly trivial), $H$ is valuative and, denoting by $v = v_H$, one has $1+\mf_v \leq T$ and $\Oc_v^\times \leq H$.

Let $\widehat T \leq \widehat K$ be given and assume that $\widehat K / \widehat T$ is torsion-free.
Assume that $\Hb(\widehat T) = \widehat T$.
Then there exists an $\widehat H$ such that $\widehat T \leq \widehat H \leq \widehat K$, $\widehat H / \widehat T$ is cyclic (possibly trivial), $\widehat H$ is valuative and, denoting by $v = v_{\widehat H}$, one has $1+\mf_v \leq \widehat T$ and $\Oc_v^\times \leq \widehat H$.
\end{prop}
\begin{proof}
This follows immediately from Remark \ref{rem:ZofT} and Theorem \ref{thm:rigid-to-valuative}.
Indeed, recall that $T = \Hb(T)$ iff $T$ is rigid and $\widehat T = \Hb(\widehat T)$ iff $\widehat T$ is rigid.
\end{proof}

\begin{prop}
\label{prop:main-kthy}
Let $K$ be an arbitrary field.
Let $-1 \in T \leq K^\times$ be given and denote by $\Hb(T) = H$.
Assume furthermore that $T \neq \Hb(T)$.
Then $H = \Hb(T)$ is valuative; denoting $v = v_H$ one furthermore has $1+\mf_v \leq T$ and $T \cdot \Oc_v^\times = H$.

Let $\widehat T \leq \widehat K$ be given such that $\widehat K / \widehat T$ is torsion-free and denote by $\widehat H = \Hb(\widehat T)$.
Assume furthermore that $\widehat T \neq \Hb(\widehat T)$.
Then $\widehat H = \Hb(\widehat T)$ is valuative; denoting $v = v_{\widehat H}$ one furthermore has $1+\mf_v \leq \widehat T$ and $\widehat H$ is the minimal submodule of $\widehat K$ such that $\widehat K/\widehat H$ is torsion-free, $\widehat T \leq \widehat H$ and $\Oc_v^\times \leq \widehat H$.
\end{prop}
\begin{proof}
Again, we show the pro-$\ell$ case as the mod-$\ell$ case is essentially identical.
For this, we only need to show that $\widehat H$ is valuative.
Indeed, the rest follows from Lemma \ref{lem:rigid-valuative-properties}, Remark \ref{rem:ZofT} and Theorem \ref{thm:rigid-to-valuative} since $\widehat T$ is the intersection of all $\widehat H'$ such that $\widehat T \leq \widehat H' \leq \widehat H$, $\widehat H / \widehat H'$ is cyclic and $\widehat K / \widehat H'$ is torsion-free.
The minimality of $\Hb(\widehat T) = \widehat H$ ensures that then $\widehat H$ is generated by $\widehat T$ and $\Oc_v^\times$ as described above (see Theorem \ref{thm:rigid-to-valuative} and/or Lemma \ref{lem:rigid-in-res-field}).

Assume, for a contradiction, that $\widehat H$ is not valuative; in particular $\widehat H \neq \widehat K$.
Take $\widehat T \leq \widehat G_1,\widehat G_2 \leq \widehat H$ such that $\widehat K/\widehat G_i$ is torsion-free, $\widehat H/\widehat G_i$ is cyclic while $\widehat G_1 \cap \widehat G_2$ is non-cyclic (such $\widehat G_i$ exist since $\widehat T \neq \widehat \Hb(T)$).
Then $\widehat G_i$ are both rigid.
Take $\widehat S_i$ such that $\widehat G_i \leq \widehat S_i \leq \widehat K$, $\widehat S_i$ is valuative, $\widehat S_i/\widehat G_i$ is cyclic and, denoting $v_i = v_{\widehat S_i}$, $1+\mf_{v_i} \leq \widehat G_i$ as in Theorem \ref{thm:rigid-to-valuative}.
Furthermore, by enlarging $\widehat S_i$ if needed, we can assume with no loss that $\widehat K/\widehat S_i$ is torsion-free and thus $\widehat S_i \cap \widehat H = \widehat G_i$ since $\widehat S_i \not\subset \widehat H$ as $\widehat H$ is non-valuative.
Since $\widehat H$ is not valuative, we further deduce from Lemma \ref{lem:non-valuative-comparable} that $v_i$ are comparable.
In particular, $\widehat S_1 \cap \widehat S_2$ is valuative and thus rigid.
We deduce from this that $\widehat H \cap \widehat S_1 \cap \widehat S_2 = \widehat G_1 \cap \widehat G_2$ is rigid by Lemma \ref{lem:rigid-valuative-properties}.
From this we deduce that $\widehat T$ is rigid using, again, Lemma \ref{lem:rigid-valuative-properties}; this is because $\widehat T$ can be written as the intersection of all possible $\widehat G_i$ as above.
However, using Lemma \ref{lem:rigid-iff-wedge} and Remark \ref{rem:ZofT}, this contradicts the fact that $\widehat T \neq \Hb(\widehat T)$ by Remark \ref{rem:ZofT} (namely, $\widehat T = \Hb(\widehat T)$ if and only if $\widehat T$ is rigid).
\end{proof}

\subsection{Restricting the Characteristic}
\label{sec:restr-char}

We would like to describe a condition ensuring that $\Char k(v) \neq \ell$ as soon as $\Char K \neq \ell$ for the valuations $v$ produced in Propositions \ref{prop:rigid-kthy} and \ref{prop:main-kthy}.
Of course, if $\Char K > 0$, then $\Char k(v) = \Char K$ and so there is nothing to prove.
On the other hand, in \S \ref{sec:compatibility-coars} (see also Example \ref{ex:pro-ell-example}), we show that valuations $v$ with $\Char k(v) = \ell$ {\bf can and do} arise in Propositions \ref{prop:rigid-kthy} and \ref{prop:main-kthy} if $\Char K = 0$.

\begin{lem}
\label{lem:pops-lemma}
Let $(K,v)$ be a valued field such that $\Char K \neq \ell$.
Denote by $L = K(\sqrt[\ell^n]{1+\mf_v})$ and $w$ a chosen prolongation of $v$ to $L$.
Let $\Delta$ be the convex subgroup of $\Gamma_v$ generated by $v(\ell)$ (this is trivial unless $\Char k(v) = \ell$).
Then $\Delta \leq \ell^n \cdot \Gamma_w$.
\end{lem}
\begin{proof}
If $\Char k(v) \neq \ell$ then $v(\ell) = 0$ and the lemma is trivial.
So assume that $\Char k(v) = \ell$.
Let $x \in K^\times$ be such that $0 < v(x) \leq v(\ell)$ and so $1+x \in L^{\times \ell^n}$.
Take $y \in L$ such that $1+x = (1+y)^{\ell^n}$.
Note that $y \in \Oc_w$ and since $1+x = (1+y)^{\ell^n} = 1+y^{\ell^n} \mod \mf_w$, we deduce that $y \in \mf_w$.
Expanding the equation $1+x = (1+y)^{\ell^n}$ we deduce that $x = \ell^n \cdot y \cdot \epsilon + y^{\ell^n}$ for some $\epsilon \in \Oc_w^\times$.
But $w(x) < w(\ell^n) < w(\ell^n \cdot y \cdot \epsilon)$ so finally $w(x) = w(y^{\ell^n})$ by the ultrametric inequality.
\end{proof}

\begin{prop}
\label{prop:rigid-kthy-with-char}
Let $K$ be an arbitrary field such that $\Char K \neq \ell$ and $\mu_\ell \subset K$.
Let $-1 \in T \leq K^\times$ be given and denote by $L = K(\sqrt[\ell]{T})$.
Assume there exists a $T' \leq L^\times$ with $L^{\times \ell} \leq T'$ such that $T' \cap K^\times = T$ and $T' = \Hb(T')$ (this implies in particular that $T = \Hb(T)$).
Then there exists an $H$ such that $T \leq H \leq K^\times$, $H/T$ is cyclic (possibly trivial) and $H$ is valuative; denoting by $v = v_H$, $H$ can be chosen so that furthermore $1+\mf_v \leq T$, $\Oc_v^\times \leq H$ and $\Char k(v) \neq \ell$.

Assume furthermore that $\mu_{\ell^\infty} \subset K$.
Let $\widehat T \leq \widehat K$ be given such that $\widehat K / \widehat T$ is torsion-free and denote by $L = K\left(\sqrt[\ell^\infty]{\widehat T \cap K^{\times}}\right)$.
Assume there exists $\widehat T' \leq \widehat L$ such that $\widehat L/\widehat T'$ is torsion-free, $\widehat T' \cap \widehat K = \widehat T$ and $\widehat T' = \Hb(\widehat T')$ (this implies in particular that $\widehat T = \Hb(\widehat T)$).
Then there exists an $\widehat H$ such that $\widehat T \leq \widehat H \leq \widehat K$, $\widehat H/\widehat T$ is cyclic (possibly trivial) and $\widehat H$ is valuative; denoting by $v = v_{\widehat H}$, $\widehat H$ can be chosen so that furthermore $1+\mf_v \leq \widehat T$, $\Oc_v^\times \leq \widehat H$ and $\Char k(v) \neq \ell$.
\end{prop}
\begin{proof}
As usual, we show the pro-$\ell$ case as the mod-$\ell$ case is similar.
By Proposition \ref{prop:rigid-kthy}, there exists $\widehat H'$ such that $\widehat T' \leq \widehat H' \leq \widehat L$, $\widehat H'/\widehat T'$ is cyclic, $\widehat H'$ is valuative and $1+\mf_{v_{\widehat H'}} \subset \widehat T'$; for simplicity, we denote by $w = v_{\widehat H'}$.
Denote by $v$ the restriction of $w$ to $K$, and denote by $\widehat H = \widehat K\cap \widehat H'$.
Our assumptions ensure that $\widehat H/\widehat T$ is cyclic, and by basic valuation theory, one has $1+\mf_{v} \subset \widehat T$ and $\Oc_{v}^\times \subset \widehat H$.
By construction, the canonical map $\widehat \Gamma_{v} / \widehat v(\widehat H) \rightarrow \widehat \Gamma_w/\widehat w(\widehat H')$ is injective.
Denote by $\Delta$ the convex subgroup of $\Gamma_v$ generated by $v(\ell)$.
By Lemma \ref{lem:pops-lemma}, $\widehat \Delta \subset \widehat v(\widehat H)$ and thus $\Delta \subset v(K^\times \cap \widehat H)$.
In particular, $\Delta$ is in the kernel of the projection $\Gamma_{v} \rightarrow \Gamma_{v_{\widehat H}}$ and thus $v_{\widehat H}(\ell) = 0$.
Since $v_{\widehat H}$ is a coarsening of $v$, one still has $1+\mf_{v_{\widehat H}} \leq \widehat T$ as required.
\end{proof}

\begin{prop}
\label{prop:main-kthy-with-char}
Let $K$ be an arbitrary field such that $\Char K \neq \ell$ and $\mu_\ell \subset K$.
Let $-1 \in T \leq K^\times$ be given and denote by $\Hb(T) = H$ and assume that $T \neq H$.
Denote by $L = K(\sqrt[\ell]{T})$ and assume that there exists $T' \leq L^\times$ such that $L^{\times \ell} \leq T'$, $T = K^\times \cap T'$ and $H = K^\times \cap \Hb(T')$.
Then $H$ is valuative; denoting $v = v_H$ one furthermore has $1+\mf_v \leq T$, $T \cdot \Oc_v^\times = H$ and $\Char k(v) \neq \ell$.

Assume furthermore that $\mu_{\ell^\infty} \subset K$.
Let $\widehat T \leq \widehat K$ be given such that $\widehat K/\widehat T$ is torsion-free and denote by $\Hb(\widehat T) = \widehat H$; assume that $\widehat T \neq \widehat H$.
Denote by $L = K\left(\sqrt[\ell^\infty]{\widehat T \cap K^\times}\right)$ and assume that there exists $\widehat T' \leq \widehat L$ such that $\widehat L/\widehat T'$ is torsion-free, $\widehat T = \widehat K \cap \widehat T'$ and $\widehat H = \widehat K\cap \Hb(\widehat T')$.
Then $\widehat H$ is valuative; denoting $v = v_{\widehat H}$ one furthermore has $1+\mf_v \leq \widehat T$, $\Char k(v) \neq \ell$.
Moreover, $\widehat H$ is the minimal submodule of $\widehat K$ such that $\widehat K/\widehat H$ is torsion-free, $\widehat T \leq \widehat H$ and $\Oc_v^\times \leq \widehat H$
\end{prop}
\begin{proof}
The proof is similar to that of Proposition \ref{prop:rigid-kthy-with-char} using Proposition \ref{prop:main-kthy} instead of Proposition \ref{prop:rigid-kthy}.
\end{proof}

\subsection{Compatibility in Coarsenings/Refinements}
\label{sec:compatibility-coars}

For a $\Z_\ell$-module $M$, denote by \[\tor(M) = \text{Tor}^1_{\Z_\ell}(M,\Q_\ell/\Z_\ell); \ \text{ this is the submodule of $\Z_\ell$-torsion in $M$.} \]
For simplicity, we denote by $M/\tor := M/\tor(M)$.
To conclude this section, we would like to describe how to precisely detect the image of $\Oc_v^\times$ and $(1+\mf_v)$ in $K^\times/\ell$ resp. $\widehat \Oc_v$ and $\widehat{1+\mf_v}$ (modulo torsion) for the valuations $v$ described in Propositions \ref{prop:rigid-kthy} and \ref{prop:main-kthy}.
Moreover, we explore the compatibility of this condition in taking coarsenings/refinements of $v$.

\begin{defn}
\label{defn:Ufs}
Let $(K,v)$ be a valued field.
We denote by $\Uf_v = \Oc_v^\times \cdot \Kl$ and $\Uf_v^1 = \langle 1+\mf_v , -1, \Kl \rangle$.
Observe that $\Uf_v/\Uf_v^1$ is isomorphic to $k(v)^\times/\ell$ via the canonical (surjective) map $\Oc_v^\times \rightarrow \Uf_v/\Uf_v^1$ if $\ell \neq 2$.
If $\ell = 2$, $\Uf_v/\Uf_v^1$ is isomorphic to $k(v) / \langle k(v)^{\times \ell}, -1 \rangle$.

We denote by $\Uc_v$ the $\ell$-adic completion of $\Oc_v^\times$.
Since $\Gamma_v$ is torsion-free, $\Uc_v$ is a sub-module of $\widehat K$ and $\widehat K / \Uc_v$ is precisely the $\ell$-adic completion of $\Gamma_v$.
In fact, the map $\widehat K \twoheadrightarrow \widehat K / \Uc_v = \widehat \Gamma_v$ is precisely the $\ell$-adic completion of the homomorphism $v : K^\times \twoheadrightarrow \Gamma_v$.

Consider the canonical map $\Uc_v \twoheadrightarrow \widehat{k(v)} \twoheadrightarrow \widehat{k(v)}/\tor$ and denote by $\Uc_v^1$ the kernel of this map.
In particular:
\begin{itemize}
\item $\widehat K / \Uc_v$ is torsion-free.
\item $\widehat K / \Uc_v^1$ and $\Uc_v / \Uc_v^1$ are torsion free.
\item There is a canonical isomorphism $\Uc_v / \Uc_v^1 \cong \widehat{k(v)}/\tor$.
\end{itemize}

For a subgroup $\Kl \leq T \leq K^\times$, $-1 \in T$, we denote by $T_v$ the image of $T \cap \Oc_v^\times$ in $k(v)^\times/\ell$.
Thus, the map $T \mapsto T_v$ is a bijection between the collection of subgroups $T$ such that $\Uf_v^1 \leq T \leq \Uf_v$ and subgroups $T_v$ such that $\langle -1, k(v)^{\times \ell}\rangle \leq T_v \leq k(v)^\times$.

For a submodule $\widehat T \leq \widehat K$ such that $\widehat K / \widehat T$ is torsion-free, we similarly denote by $\widehat T_v$ the image of $\widehat T \cap \Uc_v$ in $\widehat{k(v)}$.
One has a bijection between submodules $\widehat T \leq \widehat K$ such that $\Uc_v^1 \leq \widehat T \leq \Uc_v$ and submodules $\widehat T_v \leq \widehat {k(v)} / \tor$. 
In the following lemma, we show that this bijection respects rigidity.
\end{defn}

\begin{lem}
\label{lem:rigid-in-res-field}
Let $(K,v)$ be a valued field and let $T \leq K^\times$ be a subgroup such that $\langle \Uf_v^1,-1 \rangle \leq T \leq \Uf_v$.
Assume furthermore that $v = v_H$ for $H = \Uf_v$; equivalently, $\Gamma_v$ contains no non-trivial $\ell$-divisible convex subgroups.
The following are equivalent:
\begin{enumerate}
\item $T$ is rigid resp. valuative (as a subgroup of $K^\times/\ell$).
\item $T_v$ is rigid resp. valuative (as a subgroup of $k(v)^\times/\ell$).
\end{enumerate}

Similarly, let $\widehat T \leq \widehat K$ be given such that $\Uc_v^1 \leq \widehat T \leq \Uc_v$. Assume furthermore that $v = v_{\widehat H}$ for $\widehat H = \Uc_v$; equivalently, $\Gamma_v$ contains no non-trivial $\ell$-divisible convex subgroups.
The following are equivalent:
\begin{enumerate}
\item $\widehat T$ is rigid resp. valuative (as a submodule of $\widehat K$).
\item $\widehat T_v$ is rigid resp. valuative (as a submodule of $\widehat {k(v)}$).
\end{enumerate}
\end{lem}
\begin{proof}
We prove the pro-$\ell$ case as the mod-$\ell$ case is similar.
The fact that $\widehat T_v$ is valuative if and only if $\widehat T$ is valuative follows immediately from the definitions by taking the valuation-theoretic composition of the corresponding valuations.
Assume then that $\widehat T$ is rigid.
Take $\bar x \in k(v)^\times \smallsetminus k(v)^\times \cap \widehat T_v$ and pick a representative $x \in \Oc_v^\times$ for $\bar x$.
Observe that $x \notin \widehat T$ and $v(1-x) = 0$ for otherwise $x \in \widehat T$.
As $\widehat T$ is rigid, $\langle 1-x,x \rangle\mod \widehat T$ is cyclic in $\Uc_v / \widehat T$.
Thus, the same is true for $\langle 1-\bar x, \bar x \rangle\mod\widehat T_v$ in $\widehat k(v) / \widehat T_v$. 

Conversely, assume that $\widehat T_v$ is rigid.
Take $x \in K^\times \smallsetminus K^\times \cap \widehat T$.
If $v(x) \neq 0$, $1-x = 1 \mod \widehat T$ or $1-x = x \mod \widehat T$ since $\Uc_v^1 \leq \widehat T$.
On the other hand, if $v(x) = 0$, consider $\bar x$ the image of $x$ in $k(v)$.
Then $\langle \bar x, (1-\bar x) \rangle \mod \widehat T_v$ is cyclic so that the same is true for $\langle x, (1-x) \rangle \mod \widehat T$ since, again, $\Uc_v^1 \leq \widehat T$.
\end{proof}

\begin{defn}
\label{defn:Vc-and-Tc}
Let $K$ be a field.
We denote by $\Vc_K$ the collection of all (possibly trivial) valuations $v$ of $K$ such that:
\begin{itemize}
\item $\Gamma_v$ has no non-trivial $\ell$-divisible convex subgroups.
\item $k(v)^\times/\langle k(v)^{\times \ell}, -1 \rangle$ is non-cyclic.
\item $\Hb( \langle k(v)^{\times \ell}, -1 \rangle) = k(v)^\times$.
\end{itemize}
Moreover, we denote by $\Tc_K$ the collection of subgroups $T \leq K^\times$ such that:
\begin{itemize}
\item $\Kl \leq T$ and $-1 \in T$.
\item $T \leq \Hb(T)$ is minimal; i.e. whenever $\langle \Kl, -1 \rangle \leq T' \leq T$ and $T \neq T'$, one has $\Hb(T') \not\leq \Hb(T)$.
This implies in particular that $T = \langle \Kl, -1 \rangle$ or $\Hb(T) \neq K^\times$.
\item $T \neq \Hb(T)$.
\end{itemize}
For $v \in \Vc_K$, denote by $\Vc_K^v$ the subset of $\Vc_K$ consisting only of valuations which are finer than $v$.
Similarly, denote by $\Tc_K^v$ the subset of $\Tc_K$ consisting only of subgroups $T$ such that $\Uf_v^1 \leq T \leq \Hb(T) \leq \Uf_v$.

Denote by $\widehat \Vc_K$ the collection of all (possibly trivial) valuations $v$ of $K$ such that:
\begin{itemize}
\item $\Gamma_v$ has no non-trivial $\ell$-divisible convex subgroups.
\item $\widehat{k(v)}/\tor$ is non-cyclic.
\item $\Hb(\tor(\widehat{k(v)})) = \widehat{k(v)}$.
\end{itemize}
Moreover, we denote by $\widehat\Tc_K$ the collection of subgroups $\widehat T \leq \widehat K$ such that:
\begin{itemize}
\item $\widehat K / \widehat T$ is torsion-free.
\item $\widehat T \leq \Hb(\widehat T)$ is minimal; i.e. whenever $\widehat T' \leq \widehat T$, $\widehat K / \widehat T'$ is torsion-free and $\widehat T' \neq \widehat T$, one has $\Hb(\widehat T') \not\leq \Hb(\widehat T)$.
This implies in particular that $\widehat T = \tor(\widehat K)$ or $\Hb(\widehat T) \neq \widehat K$.
\item $\widehat T \neq \Hb(\widehat T)$.
\end{itemize}
For $v \in \widehat \Vc_K$, denote by $\widehat \Vc_K^v$ the subset of $\widehat \Vc_K$ consisting only of valuations which are finer than $v$.
Similarly, denote by $\widehat \Tc_K^v$ the subset of $\widehat \Tc_K$ consisting only of subgroups $\widehat T$ such that $\Uc_v^1 \leq \widehat T \leq \Hb(\widehat T) \leq \Uc_v$.
\end{defn}

\begin{remark}
One should note that $\Vc_K = \widehat \Vc_K$ provided that $\Char K \neq \ell$ and $\mu_{\ell^\infty} \subset K$.
Indeed, for such fields $K^\times/\ell$ is cyclic if and only if $\widehat K$ is cyclic, $-1 \in K^{\times \ell}$ and $\widehat K$ is torsion-free and the same is true replacing $K$ with $k(v)$ for all valuations $v$ of $K$.
Recall that $\widehat K / \ell = K^\times/\ell$.
If we denote by ${\bf1} \leq \widehat K$ the trivial submodule, then $K^\times \cap (\Hb({\bf 1}) \cdot \widehat K^{\ell})$ contains $\Hb(K^{\times \ell})$.
Thus, $\Hb(K^{\times \ell}) = K^\times$ implies that $\Hb({\bf 1}) = \widehat K$ since $\widehat K / \Hb({\bf 1})$ is torsion-free.

Conversely, we note that for all $x,y \in \widehat K_1^M(K)$, $\{x,y\}=0$ or $\{x,y\}$ is non-torsion (see, e.g. Proposition \ref{prop:central-kernel-pairig-with-H-2} and the remark in the proof of the pro-$\ell$ case of Lemma \ref{lem:cup-product}).
The same is true replacing $K$ with $k(v)$ for all valuations $v$ of $K$; this can be immediately deduced as $\widehat K_*^M(k(v))$ embeds into $\widehat K_*^M(K)/\Uc_v^1$ as the subring generated by products from $\widehat{k(v)} = \Uc_v/\Uc_v^1 \leq \widehat K/\Uc_v^1$.
Thus, $\Ib({\bf 1}) = \widehat K$ implies that for all $x \in \widehat K \smallsetminus \widehat K^\ell$, there exists $z \in \widehat K \smallsetminus \widehat K^\ell$ such that $x,z$ are $\Z_\ell$-independent and $\{x,z\}=0$.
We deduce that $\{x \mod \widehat K^\ell,z\mod\widehat K^\ell\} = 0$ as well, and thus $\Ib(K^{\times \ell}) = K^\times$ (see Remark \ref{rem:ZofT}).
\end{remark}

\begin{example}
\label{example:div-vals}
Let $K$ be a function field of transcendence degree $\geq 2$ over a field $k$ such that $k^\times = k^{\times \ell}$ (we make no assumptions on $\Char k$) and $k$ is relatively algebraically closed in $K$.
The {\bf prime divisors} of $K|k$ are valuations $v$ of $K$ which correspond to some Weil prime divisor on some model $X \rightarrow \Spec k$ of $K|k$.

In fact, prime divisors of $K|k$ are contained in $\Vc_K$ and in $\widehat \Vc_K$.
Let $v$ denote a prime divisor of $K|k$.
First note that $\Gamma_v \cong \Z$ so that $v$ satisfies the first property required by $\Vc_K$ resp. $\widehat \Vc_K$.
Moreover, one has $\td(k(v)|k) = \td(K|k) - 1 \geq 1$ and so $\dim_{\Z/\ell}(k(v)^\times/\ell) \geq 2$; since $K$ contains the $\ell$-closed subfield $k$, we deduce that $v$ satisfies the second condition of $\Vc_K$ and $\widehat \Vc_K$.

To simplify the notation, we denote by $F = k(v)$.
Let us now show the last condition which ensures $v \in \Vc_K$ resp. $\widehat \Vc_K$.
Let $x \in F^\times \smallsetminus F^{\times \ell}$ be given; in particular, $x$ is transcendental over $k$.
Denote by $M$ the relative algebraic closure of $k(x)$ inside $F$.
Since $M$ is a function field of transcendence degree 1 over $k$, there exists $z \in k(x)$ such that the images of $x$ and $z$ in $M^\times/\ell$ are $(\Z/\ell)$-independent.

Such a $z$ exists as follows.
First note that $x \notin F^{\times \ell}$ implies that $x$ represents a non-trivial element of $M^\times/\ell$.
Denote by $\Cc$ the unique complete normal model of $M|k$ and consider the map $\Cc \rightarrow \Pbb_k^1$ which is induced by the inclusion $k(x) \rightarrow M$.
By the approximation theorem, there exists a prime divisor $v$ of $\Pbb^1_k$ and a function $z \in k(x)$ such that $v \neq v_0, v_\infty$, $v(z) = 1$ and $v$ is unramified in the cover $\Cc \rightarrow \Pbb_k^1$; here $v_0$ resp. $v_\infty$ denotes the prime divisor associated to $0 \in \Pbb_M^1(k)$ resp. $\infty \in \Pbb_M^1(k)$.
Since $v$ is unramified, for any prime divisor $w$ of $\Cc$ which prolongs $v$ one has $w(z) = 1$ and thus $z \notin M^{\times \ell}$.
Moreover, as $v \neq v_0,v_\infty$ and the divisor associated to $x$ is precisely $v_0 - v_\infty$, we deduce that the images of $x$ and $z$ are independent in $\Div(\Cc)/\ell$.
In particular, $z,x$ are independent in $M^\times/\ell$.

Thus, the images of $x,z$ in $F^\times/\ell$ and $\widehat F$ are also independent since $M$ is relatively algebraically closed in $F$.
In particular, $\langle x,z \rangle$ also non-cyclic in $\widehat F$.
On the other hand, a classical theorem of Milnor states that one has a short exact sequence:
\[ 0 \rightarrow K_2^M(k) \rightarrow K_2^M(k(x)) \rightarrow \bigoplus_w K_1^M(k(w)) \rightarrow 0 \]
where $w$ varies over all the prime divisors of $k(x)|k$ which correspond to closed points of $\Abb_k^1$ and the rightmost map is the sum of the corresponding tame symbols.
Applying the (right exact) functor $(\bullet)\otimes_\Z \Z/\ell$ resp. $\widehat\bullet$ to this short exact sequence, we deduce that $K_2^M(k(x))/\ell = 0$ resp. $\widehat K_2^M(k(x)) = 0$ since $k(w)^\times = k(w)^{\times \ell}$ for all such $w$.
In particular, $\{x,z\} = 0 \in K_2^M(F)/\ell$ and $\{x,z\} = 0 \in \widehat K_2^M(F)$.
From this, we deduce the third condition required by $\Vc_K$ resp. $\widehat \Vc_K$ using Remark \ref{rem:ZofT}.
\end{example}

For $v \in \Vc_K$ one has a canonical map $\Vc_{k(v)} \rightarrow \Vc_K$ defined by $w \mapsto w \circ v$, where $\circ$ denotes the valuation-theoretic composition; this map is injective and its image is precisely $\Vc_K^v$.
One also has a canonical map $\Tc_{k(v)} \rightarrow \Tc_{K}$ defined by sending $T' \in \Tc_{k(v)}$ to the unique $T$ such that $\Uf_v^1 \leq T \leq \Uf_v$ and $T_v = T'$ (see definition \ref{defn:Ufs}); again, this map is injective with image $\Tc_K^v$.
For $v \in \widehat \Vc_K$ one has canonical maps $\widehat\Vc_{k(v)} \rightarrow \widehat\Vc_K$ and $\widehat\Tc_{k(v)} \rightarrow \widehat\Tc_{K}$ which are injective and whose image is precisely $\widehat\Vc_K^v$ resp. $\widehat\Tc_K^v$; these maps are defined in the obvious analogous way.

\begin{thm}
\label{thm:main-kthy-compatibility}
Let $K$ be an arbitrary field.
For any $v \in \Vc_K$, one has $\Uf_v^1 \in \Tc_K$.
The map $\Vc_K \rightarrow \Tc_K$ defined by $v \mapsto \Uf_v^1$ is a bijection.
Let $v \in \Vc_K$ be given.
Then the map $\Vc_K^v \rightarrow \Tc_K^v$ induced by $\Vc_K \rightarrow \Tc_K$ is a bijection.
Moreover, this bijection is compatible with the bijection $\Vc_{k(v)} \rightarrow \Tc_{k(v)}$ in the sense that the following diagram (of bijections) commutes:
\[
\xymatrix{
\Vc_K^v \ar[r]  & \Tc_K^v \\
\Vc_{k(v)} \ar[r] \ar[u] & \Tc_{k(v)} \ar[u]
}
\]

For any $v \in \widehat\Vc_K$, one has $\Uc_v^1 \in \widehat\Tc_K$.
The map $\widehat\Vc_K \rightarrow \widehat\Tc_K$ defined by $v \mapsto \Uc_v^1$ is a bijection.
Let $v \in \widehat\Vc_K$ be given.
Then the map $\widehat\Vc_K^v \rightarrow \widehat\Tc_K^v$ induced by $\widehat\Vc_K \rightarrow \widehat\Tc_K$ is a bijection.
Moreover, this bijection is compatible with the bijection $\widehat\Vc_{k(v)} \rightarrow \widehat\Tc_{k(v)}$ in the sense that the following diagram (of bijections) commutes:
\[
\xymatrix{
\widehat\Vc_K^v \ar[r]  & \widehat\Tc_K^v \\
\widehat\Vc_{k(v)} \ar[r] \ar[u] & \widehat\Tc_{k(v)} \ar[u]
}
\]

\end{thm}
\begin{proof}
Let $v \in \Vc_K$ be given.
By Proposition \ref{prop:main-kthy} and Lemma \ref{lem:rigid-in-res-field} and/or Theorem \ref{thm:rigid-to-valuative}, along with the properties of the valuations in $\Vc_K$, we deduce that $\Hb(\Uf_v^1) = \Uf_v$.
The conditions on $v \in \Vc_K$ ensure that $v = v_H$ where $H = \Uf_v$.
Suppose that $\langle \Kl, -1 \rangle \leq T \leq \Uf_v^1$ and $T \neq \Uf_v^1$.
We deduce from Proposition \ref{prop:main-kthy} that $\Hb(T) =: H'$ is valuative and, setting $v' = v_{H'}$, we have $1+\mf_{v'} \subset T$.
Suppose for a contradiction that $H' = \Hb(T) \leq H$.
Then $v$ is a coarsening of $v'$ and so $1+\mf_v \leq 1+\mf_{v'} \leq T' \leq \Uf_v^1$.
But this contradicts the fact that $T' \neq \Uf_v^1$.
Thus, $\Uf_v^1 \leq \Hb(\Uf_v^1) = \Uf_v$ satisfy the minimality condition required by $\Tc_K$.
This shows that the map $\Vc_K \rightarrow \Tc_K$ is well defined and the fact that $v = v_H$ for $H = \Uf_v = \Hb(\Uf_v^1)$ implies that this map is injective.

On the other hand, suppose $T \in \Tc_K$ and recall that $T \leq \Hb(T) =: H$, $T \neq H$.
By Proposition \ref{prop:main-kthy}, $H$ is valuative and, setting $v = v_H$, one has $1+\mf_v \subset T$.
In particular, $\Uf_v^1 \subset T$ and $\Uf_v \subset H$.
But the minimality condition of $T \leq \Hb(T)$ ensures that in fact $\Uf_v^1 = T$ and $\Uf_v = H$.
This shows that the map $\Vc_K \rightarrow \Tc_K$ is surjective, as required.

As for the compatibility in residue fields let $v \in \Vc_K$ be given and observe that $\Uf_v^1 \leq T \leq \Hb(\Uf_v^1) = \Uf_v$ implies that $\Uf_v^1 \leq T \leq \Hb(T) \leq \Uf_v^1$ by Lemma \ref{lem:rigid-in-res-field}.
Take $w \in \Vc_K$ a coarsening of $v$.
Denote by $v/w$ the valuation induced on $k(w)$ by $v$.
Then one has a canonical short exact sequence:
\[ 0 \rightarrow \Gamma_{v/w} \rightarrow \Gamma_v \rightarrow \Gamma_w \rightarrow 0. \]
As $\Gamma_v$ contains no non-trivial $\ell$-divisible convex subgroups, the same is true for $\Gamma_{v/w}$.
Moreover, $k(v) = k(v/w)$.
Lastly, Lemma \ref{lem:rigid-in-res-field} ensures the final condition of Definition \ref{defn:Vc-and-Tc} for $v/w$ to be in $\Vc_{k(w)}$ and the rest is easy.
The pro-$\ell$ case is virtually identical.
\end{proof}

\begin{remark}
\label{rem:T-euqals-Z-of-T}
Note that when $(K^\times/\ell)/\langle -1 \rangle$ resp. $\widehat K^\times / \tor$ is cyclic, one cannot deduce anything about the units of the valuation $v$ as above.
Consider, for example, $K = \overline \F_p^{\Z_\ell}$ where $\Z_\ell$ is considered as the $\ell$-Sylow subgroup of $G_{\F_p} = \widehat \Z = \prod_q \Z_q$.
The field $K$ does not possess any non-trivial valuations and $K^\times/\ell$ resp. $\widehat K$ is cyclic.
In particular, the only valuative subgroup of $K^\times/\ell$ resp. $\widehat K$ is precisely $K^\times/\ell$ resp. $\widehat K$.
On the other hand, consider $K = \C((t))$.
It is well known that, again, $G_K = \widehat \Z$ and thus $K^\times/\ell$ resp. $\widehat K$ is cyclic, generated by $t$.
However, the trivial subgroup of $K^\times/\ell$ resp. $\widehat K$ is valuative and the corresponding valuation is precisely the $t$-adic valuation.

On the other hand, arguing in a similar way as in Theorem \ref{thm:main-kthy-compatibility}, we can say a bit more.
Consider a subgroup $\Kl \leq T \leq K^\times$ such that $-1 \in T = \Hb(T)$ and $T \leq \Hb(T)$ is minimal in the same sense as Definition \ref{defn:Vc-and-Tc}.
Then there exists a valuation $v$ of $K$ such that $\Gamma_v$ contains no non-trivial $\ell$-divisible convex subgroups, $\dim_{\Z/\ell}((k(v)^\times/\ell)/\langle -1 \rangle) \leq 1$ and $T = \Uf_v^1$.
Unfortunately, one cannot detect $\Uf_v$ precisely using the K-theoretic method above since $\Uf_v/\Uf_v^1 \cong (k(v)^\times/\ell)/\langle -1 \rangle$ is cyclic.
However one can say that $\Uf_v$ is precisely the {\bf minimal valuative} subgroup of $K^\times$ which contains $\Uf_v^1$.

Similarly one can consider $\widehat T \leq \widehat K$ such that $\widehat K/\widehat T$ is torsion-free and $\widehat T = \Hb(\widehat T)$ is minimal in the sense of Definition \ref{defn:Vc-and-Tc}.
Then there exists a valuation $v$ of $K$ such that $\Gamma_v$ contains no non-trivial $\ell$-divisible convex subgroups, $\widehat{k(v)} / \tor$ is cyclic and $\widehat T = \Uc_v^1$.
As above, unfortunately, one cannot detect $\Uc_v$ precisely using the $\ell$-adically complete K-theoretic method above since $\widehat{k(v)} / \tor$ is cyclic.
\end{remark}

\section{Almost Commuting Liftable Subgroups of Galois Groups}
\label{sec:ACL-groups}

Throughout this section, unless otherwise noted, $K$ will denote a field such that $\Char K \neq \ell$ and $\mu_{\ell} \subset K$.
Throughout, we fix once and for all an isomorphism of $G_K$-modules $\mu_\ell \cong \Z/\ell$, and we denote by $\omega = \omega_\ell$ the element of $\mu_\ell$ which corresponds to $1 \in \Z/\ell$ under the isomorphism $\Z/\ell \cong \mu_\ell$.
When we make the assumption that $\mu_{\ell^\infty} \subset K$, we implicitly also fix, in this case, an isomorphism of $G_K$-modules $\lim_m \mu_{\ell^m} \cong \Z_\ell$ obtained by choosing a compatible system $(\omega_{\ell^m})_m$ with $\omega_m$ a generator of $\mu_{\ell^m}$.

We denote $K' = K(\sqrt[\ell]{K^\times})$ and $\Gc_K^a := \Gal(K'|K)$ (this notation is compatible with that which is used throughout the paper).
Recall that the Kummer pairing $K^\times/\ell \times \Gc_K^a \rightarrow \mu_\ell$ is defined by $(x,\sigma) \mapsto \sigma \sqrt[\ell]{x}/\sqrt[\ell]{x}$.
This is a perfect pairing which induces a bijection between subgroups $T$ such that $\Kl \leq T \leq K^\times$ and subgroups $A \leq \Gc_K^a$.
For a subgroup $T$ as above, we denote by $T^\perp = \Gal(K'|K(\sqrt[\ell]{T}))$; this is precisely the orthogonal of $T$ with respect to the pairing above.
Similarly, given $A \leq \Gc_K^a$, we denote by $A^\perp = ((K')^A)^{\times \ell} \cap K^\times$; this is precisely the orthogonal of $A$ with respect to the pairing above, and is the unique subgroup $T$ such that $\Kl \leq T \leq K^\times$ and $A = \Gal(K'|K(\sqrt[\ell]{T}))$.

For a field $K$ such that $\mu_{\ell^\infty} \subset K$, we denote by $K^{ab} = K(\sqrt[\ell^\infty]{K})$ and $\Pi_K^a = \Gal(K^{ab}|K)$ the maximal pro-$\ell$ abelian Galois group of $K$ (this notation is compatible with that which is used throughout the paper).
Consider $\Z_\ell(1) = \lim_{n} \mu_{\ell^n}$; one has a canonical perfect Kummer pairing $\widehat K \times \Pi_K^a \rightarrow \Z_\ell(1)$ defined by taking the limit of the corresponding Kummer pairings $K^\times/\ell^n \times \Pi_K^a / \ell^n \rightarrow \mu_{\ell^n}$, $(x,\sigma) \mapsto \sigma \sqrt[\ell^n]{x}/\sqrt[\ell^n]{x}$, for all $n$.
For $\widehat A \leq \Pi_K^a$ resp. $\widehat T \leq \widehat K$ we denote by $\widehat A^\perp \leq \widehat K$ resp. $\widehat T^\perp \leq \Pi_K^a$ the corresponding Kummer orthogonals with respect to this pairing.
This induces a 1-1 correspondence between subgroups $\widehat A \leq \Pi_K^a$ such that $\Pi_K^a / \widehat A$ is torsion-free and submodules $\widehat T \leq \widehat K$ such that $\widehat K/\widehat T$ is torsion-free.

The main results of \S~\ref{sec:rigid-elements} allow one to detect the valuations $v$ of $K$ using the Milnor-K-theory of the field.
In the mod-$\ell$ case, one uses the information encoded in $K_1^M(K)/\ell$, $K_2^M(K)/\ell$ and the product $(K_1^M(K)/\ell)^{\otimes 2} \rightarrow K_2^M(K)/\ell$.
In the pro-$\ell$ case, one uses the information encoded in $\widehat K_1^M(K)$, $\widehat K_2^M(K)$ and the product $(\widehat K_1^M(K))^{ \otimes 2} \rightarrow \widehat K_2^M(K)$.

If the field $K$ has characteristic different from $\ell$, the Merkurjev-Suslin theorem shows that this information is already encoded in $H^1(K,\mu_\ell)$, $H^2(K,\mu_\ell^{\otimes 2})$ and the cup product $(H^1(K,\mu_\ell))^{\otimes 2}\rightarrow H^2(K,\mu_\ell^{\otimes 2})$ for the mod-$\ell$ case; for the pro-$\ell$ case one can use $H^1(K,\Z_\ell(1))$, $H^2(K,\Z_\ell(2))$ and the cup product $(H^1(K,\Z_\ell(1)))^{\otimes 2}\rightarrow H^2(K,\Z_\ell(2))$.
Another consequence of the Merkurjev-Suslin theorem shows that this datum can be deduced from the cohomology of small pro-$\ell$ Galois groups ($\Gc_K^c$ resp. $\Pi_K^c$) over $K$ in the presence of enough roots of unity (see, e.g., \cite{Chebolu2009} for the explicit details).
So, in some sense, we've already produced group theoretical recipes to detect valuations from small Galois groups.
In the context of producing precise recipes for the local theory in birational anabelian-geometry, this is somewhat unsatisfactory.
Thus, in this section, we translate the results of \S~\ref{sec:rigid-elements} into {\bf explicit conditions using the group-theoretical structure} of $\Gc_K^c$ resp. $\Pi_K^c$ instead of the structure of its cohomology; these conditions are inspired by Bogomolov and Tschinkel's theory of commuting-liftalble pairs \cite{Bogomolov2007}.

\subsection{Hilbert Decomposition Theory}
\label{sec:hilb-decomp-theory}

Let $(K,v)$ be a valued field such that $\Char K \neq \ell$ and $\mu_\ell \subset K$.
Pick a prolongation $v'$ of $v$ to $K' = K(\sqrt[\ell]{K})$.
We denote by $D_v = D_{v'|v}$ resp. $I_v = I_{v'|v}$ the decomposition and inertia subgroups of $v'|v$ inside $\Gc_K^a = \Gal(K'|K)$.
Note that as $\Gc_K^a$ is abelian, the subgroups $I_v \leq D_v$ are independent of choice of prolongation $v'$.

Suppose furthermore $\mu_{\ell^\infty} \subset K$.
Pick a prolongation $v^{ab}$ of $v$ to $K^{ab} = K(\sqrt[\ell^\infty]{K})$.
Denote by $\widehat D_v = D_{v^{ab}|v}$ and $\widehat I_v = I_{v^{ab}|v}$ the decomposition and inertia subgroups of $v^{ab}|v$ inside $\Pi_K^a = \Gal(K^{ab}|K)$.
Again, since $\Pi_K^a$ is abelian, these are independent of choice of prolongation $v^{ab}|v$.

In fact, if $\Char k(v) \neq \ell$, we can explicitly describe these subgroups via the Kummer pairing $K^\times/\ell \times \Gc_K^a \rightarrow \mu_\ell$ resp. $\widehat K \times \Pi_K^a \rightarrow \Z_\ell(1)$ (Proposition \ref{prop:decomp-thy-prop}).
Before we prove this proposition, we first review some basic facts from Hilbert decomposition theory for valued field $(K,v)$ such that $\Char k(v) \neq \ell$ and $\mu_\ell \subset K$ resp. $\mu_{\ell^\infty} \subset K$.

Assume, then, that $\Char k(v) \neq \ell$ and let $L|K$ be an arbitrary pro-$\ell$ Galois extension ($K \subset L \subset K(\ell)$) and pick a prolongation $w$ of $v$ to $L$.
We denote by $I_{w|v}$ resp. $D_{w|v}$ the inertia resp. decomposition subgroups of $w|v$ in $\Gal(L|K)$.
One has a canonical short exact sequence:
\[ 1 \rightarrow I_{w|v} \rightarrow D_{w|v} \rightarrow \Gal(k(w)|k(v)) \rightarrow 1; \]
recall that this short exact sequence is split if $L = K(\ell)$ and that $k(w) = k(v)(\ell)$ in this case.
Moreover, we have a perfect pairing which is compatible with the action of $\Gal(k(w)|k(v))$ on $I_{w|v}$:
\[ I_{w|v} \times (\Gamma_w/\Gamma_v) \rightarrow \mu_{\ell^\infty}(k(w)) = \mu_{\ell^\infty} \cap k(v)^\times \]
defined by $(\sigma,w(x)) \mapsto \overline{\sigma x/x}$ where $y \mapsto \bar y$ is the canonical map $\Oc_w^\times \twoheadrightarrow k(w)^\times$.
To simplify the notation, we denote by $\mu_{\ell^v} := \mu_{\ell^\infty}(k(v)) = \mu_{\ell^\infty} \cap k(v)^\times$.
This pairing is compatible with the action of $\Gal(k(w)|k(v))$ on $I_{w|v}$; in particular $\Gal(k(w)|k(v))$ acts on $I_{w|v}$ via the cyclotomic character $\Gal(k(w)|k(v)) \twoheadrightarrow \Gal(k(v)(\mu_{\ell^w})|k(v))$.

If we have a tower of pro-$\ell$ Galois extensions of valued fields: $(K,v) \subset (L,w) \subset (F,w')$ then the corresponding pairings are compatible.
I.e. the following diagram is commutative in the natural sense:
\[
\xymatrix{
I_{w'|v} \ar@{->>}[d]  &*-<45pt>{\times}& (\Gamma_{w'}/\Gamma_v) \ar[r] & \mu_{\ell^{w'}} \\
I_{w|v} &*-<45pt>{\times}&  (\Gamma_{w}/\Gamma_v)  \ar[u] \ar[r] & \mu_{\ell^w} \ar[u] \\
}
\]
Moreover, the two pairings are compatible with the action of $\Gal(k(w')|k(v))$ on $I_{w'|v}$ resp. $\Gal(k(w)|k(v))$ on $I_{w|v}$. 
I.e. the surjective map $I_{w'|v} \twoheadrightarrow I_{w|v}$ is $(\Gal(k(w')|k(v)))$-equivariant; here $\Gal(k(w')|k(v))$ acts on $I_{w|v}$ via the projection $\Gal(k(w')|k(v))$ onto $\Gal(k(w)|k(v))$.

The proof of the following proposition can be found in \cite{Pop2010} Fact 2.1 in the pro-$\ell$ case and in \cite{Pop2011a} in the mod-$\ell$ case, but is explicitly stated for valuations $v$ such that $\Char k(v) \neq \ell$.
It turns out that the same proof works, at least in one direction, even if $\Char k(v) = \ell$ and we summarize this in the proposition below.
%The argument given for the case where $\Char k(v) = \ell$ uses a discussion from \cite{Pop2010b} given in the proof of Lemma 2.3(2) from loc.cit.. 

\begin{prop}
\label{prop:decomp-thy-prop}
Let $(K,v)$ be a valued field such that $\Char K \neq \ell$ and $\mu_\ell \subset K$.
Recall that $D_v^1 = \Gal(K'|K(\sqrt[\ell]{1+\mf_v}))$ and $I_v^1 = \Gal(K'|K(\sqrt[\ell]{\Oc_v^\times}))$.
Then $D_v^1 \leq D_v$ and $I_v^1 \leq I_v$.
If furthermore $\Char k(v) \neq \ell$ then $D_v^1 = D_v$ and $I_v^1 = I_v$.

Suppose furthermore that $\mu_{\ell^\infty} \subset K$.
Recall that $\widehat D_v^1 = \Gal\left(K^{ab}|K\left(\sqrt[\ell^\infty]{1+\mf_v}\right)\right)$ and $\widehat I_v^1 = \Gal\left(K^{ab}|K\left(\sqrt[\ell^\infty]{\Oc_v^\times}\right)\right)$.
Then $\widehat D_v^1 \leq \widehat D_v$ and $\widehat I_v^1 \leq \widehat I_v$.
If furthermore $\Char k(v) \neq \ell$ then $\widehat D_v = \widehat D_v^1$ and $\widehat I_v = \widehat I_v^1$.
\end{prop}
\begin{proof}
The pro-$\ell$ case is precisely \cite{Pop2010} Fact 2.1 if $\Char k(v) \neq \ell$.
If $\Char k(v) = \ell$, we note that the required direction in the proof of loc.cit. still holds.
Below we sketch the mod-$\ell$ adaptation as the pro-$\ell$ case is virtually identical.

Suppose $a \in K^\times$ is such that $\sqrt[\ell]{a} \in (K')^{D_v}$ and denote by $w$ a prolongation of $v$ to $(K')^{D_v}$.
Since $\Gamma_w = \Gamma_v$, there exists $y \in K^\times$ such that $v(a) = \ell \cdot v(y)$.
Moreover, as $k(v) = k(w)$, there exists $z \in \Oc_v^\times$ such that $\sqrt[\ell]{a}/y \in z \cdot (1+\mf_w)$.
Namely, $a/(yz)^\ell \in (1+\mf_v)$ so that $\sqrt[\ell]{a} \in K(\sqrt[\ell]{1+\mf_v})$.
Thus, $D_v^1 \leq D_v$ since $(K')^{D_v} \subset K(\sqrt[\ell]{1+\mf_v})$.
The proof that $(K')^{I_v} \subset K(\sqrt[\ell]{\Oc_v^\times})$ is similar.

Assume furthermore that $\Char k(v) \neq \ell$.
Let $(K^Z,v)$ be some Henselization of $(K,v)$; recall that $K^Z \cap K' = (K')^{D_v}$.
Let $a \in 1+\mf_v$ be given.
The polynomial $X^\ell - a$ reduces mod $\mf_v$ to $X^\ell - 1$. Since $\Char k(v) \neq \ell$ one has $\mu_\ell \subset k(v)$ and this polynomial has $\ell$ unique roots in $k(v)$.
Namely, $X^\ell-a$ has a root in $K^Z \cap K' = (K')^{D_v}$.
By Hensel's lemma, $K(\sqrt[\ell]{1+\mf_v}) \subset (K')^{D_v}$.
The proof that $K(\sqrt[\ell]{\Oc_v^\times}) \subset (K')^{I_v}$ is similar.
\end{proof}

\begin{remark}
\label{rem:charkv-is-ell}
If $\Char K \neq \ell$, $\mu_\ell \subset K$ and $\Char k(v) = \ell$, one has $I_v^1 \leq D_v^1 \leq I_v$.
This can be deduced in a similar way to \cite{Pop2010b} Lemma 2.3(2); we sketch the argument below.
Denote by $\lambda = \omega - 1 \in K$ where $\omega = \omega_\ell$ is our fixed generator of $\mu_\ell$ and recall that $v(\lambda) > 0$ since $\Char k(v) = \ell$.
Let $u \in \Oc_v^\times$ be given and set $u' = \lambda^\ell \cdot u + 1 \in 1+\mf_v$.
Then the extension of $K$ corresponding to the equation $X^\ell - u'$ is precisely the same as the extension of $K$ corresponding to the equation $Y^\ell - Y + \lambda \cdot f(Y) = u$ for some (explicit) polynomial $f(Y)$ -- this is done by making the change of variables $X = \lambda Y + 1$.
On the other hand, the maximal $(\Z/\ell)$-elementary abelian Galois extension of $k(v)$ is the extension of $k(v)$ generated by roots of polynomials of the form $X^\ell - X = \bar u$ for $\bar u \in k(v)$.
Thus, the maximal $(\Z/\ell)$-elementary abelian Galois extension of $k(v)$ is a subextension of the residue extension corresponding to $K(\sqrt[\ell]{1+\mf_v})|K$.
This immediately implies that $I_v^1 \leq D_v^1 \leq I_v$ as required.
\end{remark}

\subsection{Generalities on the Cohomology of Pro-$\ell$ Groups}
\label{sec:gener-cohom-pro-ell}

In the context of profinite groups, a ``subgroup'' will always mean a closed subgroup and a ``homomorphism'' will always mean a continuous homomorphism.
If $\Gc$ is a profinite group and $\Sc$ is a subset of $\Gc$, we will use the notation $\langle \Sc \rangle$ for the {\bf closure} of the subgroup generated by $\Sc$.
Whenever we have a surjective map $\Gc \twoheadrightarrow \Hc$ of profinite groups, and a subset $\Sc$ of $\Hc$ which converges to $1$, we will use the term ``pick lifts'' of $\Sc$ to mean: (1) take a continuous section $\widetilde{(\bullet)} : \Hc \hookrightarrow \Gc$ of $\Gc \twoheadrightarrow \Hc$ and (2) take $\widetilde \Sc$ as our set of lifts; this ensures that $\widetilde \Sc$ also converges to $1$ in $\Gc$.

Let $\Gc$ be a pro-$\ell$ group.
We will use the notation $H^i(\Gc) := H^i(\Gc,\Z/\ell)$ and $\widehat H^i(\Gc) := H^i(\Gc,\Z_\ell)$ for simplicity; here $H^*(\Gc,A)$ denotes the continuous-cochain cohomology of $\Gc$ with values in $A$.
We also denote by $\beta : H^1(\Gc) \rightarrow H^2(\Gc)$ the Bockstein map; i.e. this is the connecting homomorphism associated to the short exact sequence of trivial $\Gc$-modules: 
\[ 0 \rightarrow \Z/\ell \rightarrow \Z/\ell^2 \rightarrow \Z/\ell \rightarrow 0. \]

Recall the definition of the $(\Z/\ell)$-central descending series of $\Gc$:
\[ \Gc^{(1)} = \Gc, \ \ \Gc^{(i+1)} = [\Gc, \Gc^{(i)}] \cdot (\Gc^{(i)})^\ell. \]
For simplicity, we denote $\Gc^a = \Gc/\Gc^{(2)}$ and $\Gc^c = \Gc/\Gc^{(3)}$.
The subgroup $\Gc^{(2)}$ is precisely the Frattini subgroup of $\Gc$.
Thus, any minimal generating set $(\sigma_i)_{i \in I}$ of $\Gc^a$ which converges to $1$ yields a minimal generating set $(\sigma_i)_{i \in I}$ of $\Gc$ which converges to $1$ by choosing a continuous section $\Gc^a \hookrightarrow \Gc$ of $\Gc \twoheadrightarrow \Gc^a$; thus, in turn, we obtain a free presentation $S \twoheadrightarrow \Gc$ where $S$ is the free pro-$\ell$ group of $(\sigma_i)_{i \in I}$ and the map $S \rightarrow \Gc$ induces an isomorphism $S^a \xrightarrow{\cong} \Gc^a$.

We also recall the definition of the usual central descending series of $\Gc$:
\[ \Gc^{(1,\infty)} := \Gc, \ \ \Gc^{(i+1,\infty)} = [\Gc,\Gc^{(i,\infty)}]. \]
For simplicity we denote $\Gc^{ab} := \Gc/\Gc^{(2,\infty)}$ and $\Gc^C = \Gc/\Gc^{(3,\infty)}$.
If $\Gc^{ab}$ is torsion-free, the free presentation $S \twoheadrightarrow \Gc$ from above can be picked so that also $S^{ab} \xrightarrow{\cong} \Gc^{ab}$ is an isomorphism.

\begin{defn}
\label{defn:commutator-and-beta}
Let $\Gc$ be a pro-$\ell$ group.
For $\sigma,\tau \in \Gc^a$, we denote by $[\sigma,\tau] = \tilde\sigma^{-1} \tilde\tau^{-1} \tilde\sigma\tilde\tau$ where $\tilde\sigma,\tilde\tau \in \Gc^c$ denote some lifts of $\sigma,\tau \in \Gc^a$ under the canonical projection $\Gc^c \twoheadrightarrow \Gc^a$.
We denote by $\sigma^\ell = \tilde\sigma^\ell$ where, again, $\tilde\sigma \in \Gc^c$ is some lift of $\sigma \in \Gc^a$.
To simplify the exposition, if $\ell = 2$, we denote by $\sigma^\beta = 0$ and if $\ell \neq 2$ we denote by $\sigma^\beta = \sigma^\ell$.

Since $\Gc^{(2)}/\Gc^{(3)}$ is central in $\Gc^c$ and is killed by $\ell$, the elements $[\sigma,\tau]$ and $\sigma^\ell$ are independent of choice of lifts $\tilde\sigma,\tilde\tau$.
Thus, one has well defined maps:
\[ [\bullet,\bullet] : \Gc^a \times \Gc^a \rightarrow \Gc^{(2)}/\Gc^{(3)}, \ \ (\bullet)^\ell : \Gc^a \rightarrow \Gc^{(2)}/\Gc^{(3)}, \ \ (\bullet)^\beta : \Gc^a \rightarrow \Gc^{(2)}/\Gc^{(3)}. \]

We have a similarly defined map $\Gc^{ab} \times \Gc^{ab} \rightarrow \Gc^{(2,\infty)}/\Gc^{(3,\infty)}$ which we also denote by $[\bullet,\bullet]$ when no confusion is possible.
\end{defn}

The following fact follows from \cite{Neukirch2008} Proposition 3.8.3:

\begin{fact}
\label{fact:bilinear-linear}
For any prime $\ell$, the map $[\bullet,\bullet] : \Gc^a \times \Gc^a \rightarrow \Gc^{(2)}/\Gc^{(3)}$ is $(\Z/\ell)$-bilinear and the map $[\bullet,\bullet] : \Gc^{ab} \times \Gc^{ab} \rightarrow \Gc^{(2,\infty)}/\Gc^{(3,\infty)}$ is $\Z_\ell$-bilinear.
Moreover, the map $(\bullet)^\beta$ is $(\Z/\ell)$-linear (recall this is forced to be the trivial map if $\ell = 2$).
\end{fact}

Let $\Gc$ be a pro-$\ell$ group and pick a free presentation $S \twoheadrightarrow \Gc$, where $S$ is a free pro-$\ell$ group, such that the induced map $S^a \rightarrow \Gc^a$ is an isomorphism.
Denote by $T$ the kernel of $S \twoheadrightarrow \Gc$ and consider the spectral sequence associated to the extension $1 \rightarrow T \rightarrow S \rightarrow \Gc \rightarrow 1$:
\[ H^p(\Gc,H^q(T)) \Rightarrow H^{p+q}(S). \]
Since $S$ and $T$ have $\ell$-cohomological dimension $\leq 1$ and $S^a \xrightarrow{\cong} \Gc^a$ is an isomorphism, we deduce that the differential $d_2 : E_2^{0,1} \rightarrow E_2^{2,0}$ is an isomorphism; namely, $d_2 : H^1(T)^{\Gc} \xrightarrow{\cong} H^2(\Gc)$.
On the other hand, $H^1(T)^{\Gc} = \Hom_\Gc(T,\Z/\ell) = \Hom(T/([S,T] \cdot T^\ell), \Z/\ell)$.
Thus, $d_2$ induces a pairing
\[ T \times H^2(\Gc) \rightarrow \Z/\ell \]
whose right kernel is trivial and whose left kernel is precisely $[S,T] \cdot T^\ell$.
Thus, we obtain a perfect pairing
\[ \left(\frac{T}{[S,T] \cdot T^\ell}\right) \times H^2(\Gc) \rightarrow \Z/\ell \]
defined by $(t,\xi) \mapsto (d_2^{-1}(\xi))(t)$.

Let $\Gc$ be a torsion-free pro-$\ell$ group such that $\Gc^{ab}$ is torsion-free (e.g. $\Gc = \Gc_K$ for some field $K$ with $\Char K \neq \ell$ and $\mu_{\ell^\infty} \subset K$). 
We can, in this case, pick a presentation $S \rightarrow \Gc$ so that $S^{ab} \rightarrow \Gc^{ab}$ is an isomorphism.
We denote the kernel of this map by $\widehat T$.
Similarly to above, $d_2 : \widehat H^1(\widehat T)^{\Gc} \rightarrow \widehat H^2(\Gc)$ is an isomorphism where $d_2$ denotes the differential in the spectral sequence:
\[ \widehat H^p(\Gc, \widehat H^q(T)) \Rightarrow \widehat H^{p+q}(S). \]
We therefore obtain a pairing:
\[ \widehat T \times \widehat H^2(\Gc) \rightarrow \Z_\ell \]
whose right kernel is trivial and whose left kernel is precisely $[\widehat T,S]$.
Thus, we obtain a perfect pairing
\[ \left(\frac{\widehat T}{[\widehat T,S]} \right) \times \widehat H^2(\Gc) \rightarrow \Z_\ell \]
defined by $(t,\xi) \mapsto d_2^{-1}(\xi)(t)$.
In the next proposition, we provide a complete and explicit description of these pairings in certain cases which apply and, in particular, when $\Gc$ is the maximal pro-$\ell$ Galois group of a field $K$ with $\Char K \neq \ell$ and $\mu_\ell \subset K$ resp $\mu_{\ell^\infty}$.

\begin{prop}
\label{prop:central-kernel-pairig-with-H-2}
Let $S$ be a free pro-$\ell$ group on generators $(\sigma_i)_{i \in I}$ and give $I$ a total ordering.
Then any element $\bar \rho \in S^{(2)}/S^{(3)}$ has a unique representation as:
\[\bar \rho = \prod_{i < j} [\sigma_i,\sigma_j]^{a_{ij}(\bar \rho)} \cdot \prod_r (\sigma_r^\ell)^{b_r(\bar\rho)} \]
where $\sigma_i$ are considered as elements of $S^a$.
Any element $\bar \delta \in S^{(2,\infty)}/S^{(3,\infty)}$ has a unique representation as:
\[ \bar \delta = \prod_{i < j} [\sigma_i,\sigma_j]^{\widehat a_{ij}(\bar \delta)} \]
where $\sigma_i$ are considered as elements of $S^{ab}$.

Let $\Gc$ be a pro-$\ell$ group and pick a free presentation $S \twoheadrightarrow \Gc$ which induces an isomorphism $S^a \rightarrow \Gc^a$ where $S$ is a free pro-$\ell$-group on $(\sigma_i)_{i \in I}$.
We again endow $I$ with a total ordering.
Denote by $T$ the kernel of the homomorphism $S \twoheadrightarrow \Gc$ and consider the canonical pairing:
\[ (\bullet,\bullet) : \left(\frac{T}{[S,T] \cdot T^\ell}\right) \times H^2(\Gc) \rightarrow \Z/\ell. \]
Denote by $(x_i)_{i \in I}$ the dual basis of $H^1(\Gc)$ associated to $(\sigma_i)_{i \in I}$.
Then
\begin{itemize}
\item $(t,x_i \cup x_j) =  - a_{ij}(t)$ for $i < j$.
\item $(t,\beta x_r) = -b_r(t)$.
\end{itemize}

Assume that the inflation map $H^2(\Gc^a) \rightarrow H^2(\Gc)$ is surjective and denote by $R$ the kernel of the induced surjective map $S^c \twoheadrightarrow \Gc^c$.
Then the canonical map $T / ([S,T] \cdot T^\ell) \rightarrow R$ is an isomorphism.
In particular, we obtain an induced perfect pairing: 
\[ (\bullet,\bullet) : R \times H^2(\Gc) \rightarrow \Z/\ell \]
such that
\begin{itemize}
\item $(\bar\rho, x_i \cup x_j) = -a_{ij}(\bar \rho)$ if $i < j$, $i,j \in I$.
\item $(\bar\rho, \beta x_r) = -b_r(\bar \rho)$ for $r \in I$.
\end{itemize}

Suppose furthermore that $\Gc^{ab}$ is torsion-free.
Pick $S \rightarrow \Gc$ so that furthermore $S^{ab} \rightarrow \Gc^{ab}$ is an isomorphism.
Denote by $\widehat T$ the kernel of the homomorphism $S \twoheadrightarrow \Gc$ and consider the canonical pairing:
\[ (\bullet,\bullet) : \left(\frac{\widehat T}{[S,\widehat T]}\right) \times \widehat H^2(\Gc) \rightarrow \Z_\ell. \]
Denote by $(x_i)_{i \in I}$ the dual basis of $\widehat H^1(\Gc)$ associated to $(\sigma_i)_{i \in I}$.
Then $(t,x_i \cup x_j) = -\widehat a_{ij}(t)$ for $i < j$.

Assume moreover that the inflation map $\widehat H^2(\Gc^{ab}) \rightarrow \widehat H^2(\Gc)$ is surjective and denote by $\widehat R$ the kernel of the induced surjective map $S^C \twoheadrightarrow \Gc^C$.
Then the canonical map $\widehat T / [S,\widehat T] \rightarrow \widehat R$ is an isomorphism.
In particular, we obtain an induced perfect pairing:
\[ (\bullet,\bullet) : \widehat R \times \widehat H^2(\Gc) \rightarrow \Z_\ell \]
such that $(\bar \delta, x_i \cup x_j) = -\widehat a_{ij}(\bar \delta)$ if $i < j$, $i,j \in I$.
\end{prop}
\begin{proof}
Below we show the mod-$\ell$ case as the pro-$\ell$ case is virtually identical replacing $[S,T] \cdot T^\ell$ with $[S,\widehat T]$, and making the obvious changes in the notation.
Denote by $T$ the kernel of $S \twoheadrightarrow \Gc$.
Using a combination of \cite{Neukirch2008} Proposition 3.9.13 and 3.9.14, along with a standard limit argument in the case where $I$ is infinite, we only need to show that the canonical surjective map $T/([S,T] \cdot T^\ell) \rightarrow R$ is an isomorphism (see also the discussion immediately preceding this proposition).

Consider the free presentation $S \twoheadrightarrow \Gc \twoheadrightarrow \Gc^a$.
The differential $d_2$ in the spectral sequence associated to this extension induces a perfect pairing 
\[ S^{(2)}/S^{(3)} \times H^2(\Gc^a) \rightarrow \Z/\ell \]
which is compatible in the natural sense with the perfect pairing associated to $\Gc$.
I.e. the dual of the inflation map $H^2(\Gc^a) \rightarrow H^2(\Gc)$ is precisely the canonical map $T/([S,T] \cdot T^\ell) \rightarrow S^{(2)}/S^{(3)}$.
This canonical map is injective by our assumption that $H^2(\Gc^a) \rightarrow H^2(\Gc)$ is surjective.
We deduce that $T \cap S^{(3)} = [S,T] \cdot T^\ell$.
On the other hand, $R = (T \cdot S^{(3)})/S^{(3)}$ and so the kernel of the canonical surjective map $T \rightarrow R$ is precisely $[S,T] \cdot T^\ell$, as required.
\end{proof}

\begin{remark}
\label{rem:merk-sus}
Let $K$ be a field such that $\Char K \neq \ell$ and $\mu_{2\ell} \subset K$ (in particular, $G_K$ is torsion-free).
By the Merkurjev-Suslin theorem, the cup product $H^1(K,\mu_\ell) \otimes H^1(K,\mu_\ell) \xrightarrow{\cup} H^2(K,\mu_\ell^{\otimes 2})$ is surjective.
Since $\mu_\ell \cong \Z/\ell$, we deduce that $\Gc_K$ satisfies the assumptions of Proposition \ref{prop:central-kernel-pairig-with-H-2} in the mod-$\ell$ case.
Similarly, if $\mu_{\ell^\infty} \subset K$, $\Gc_K$ satisfies the assumptions of Proposition \ref{prop:central-kernel-pairig-with-H-2} in the pro-$\ell$ case since $H^1(K,\Z_\ell(1)) \otimes H^1(K,\Z_\ell(1)) \xrightarrow{\cup} H^2(K,Z_\ell(2))$ is also surjective.
\end{remark}

\begin{lem}
\label{lem:cup-product}
Let $\Gc$ be a pro-$\ell$ group.
Let $f,g \in \Hom(\Gc,\Z/\ell) = H^1(\Gc^a) = H^1(\Gc^c) = H^1(\Gc)$ be given and $\Z/\ell$-independent.
The following are equivalent:
\begin{enumerate}
\item $f \cup g \neq 0 \in H^2(\Gc)$.
\item $f \cup g \neq 0 \in H^2(\Gc^c)$.
\item $(\ker f)^{(2)} \cdot (\ker g)^{(2)} = \Gc^{(2)}$.
\end{enumerate}

On the other hand, assume that $\Gc^{ab}$ is torsion-free and let $f,g \in \Hom(\Gc,\Z_\ell) = \widehat H^1(\Gc^{ab}) = \widehat H^1(\Gc^C) = \widehat H^1(\Gc)$ be given and $\Z_\ell$-independent.
The following are equivalent:
\begin{enumerate}
\item $f \cup g \neq 0 \in \widehat H^2(\Gc)$.
\item $f \cup g \neq 0 \in \widehat H^2(\Gc^C)$.
\item $(\ker f)^{(2,\infty)} \cdot (\ker g)^{(2,\infty)} = \Gc^{(2,\infty)}$.
\end{enumerate}
\end{lem}
\begin{proof}
The equivalence of $(1)$ and $(3)$ is well known in the mod-$\ell$ case -- see for example \cite{Linnell1991} Theorem 1 to deduce this for the mod-$\ell$ case (after a limit argument).
Alternatively, this can be deduced from Proposition \ref{prop:central-kernel-pairig-with-H-2} along with the discussion which precedes it (we use this extensively with no mention in the proof below).

Clearly, $(1) \Rightarrow (2)$.
Let us show $(2) \Rightarrow (1)$.
Take $f,g \in H^1(\Gc) = H^1(\Gc^c) = H^1(\Gc^a)$ and suppose that $f \cup g$ is in the kernel of the inflation map $H^2(\Gc^a) \rightarrow H^2(\Gc)$.
Considering the spectral sequence associated to the extension $1 \rightarrow \Gc^{(2)} \rightarrow \Gc \rightarrow \Gc^a \rightarrow 1$, we deduce that there exists $\xi \in H^1(\Gc^{(2)})^\Gc$ such that $d_2(\xi) = f \cup g$.
On the other hand, $H^1(\Gc^{(2)}/\Gc^{(3)})^\Gc \cong H^1(\Gc^{(2)})^{\Gc^a}$ via the inflation map.
Thus, considering the spectral sequence associated to the extension $1 \rightarrow \Gc^{(2)}/\Gc^{(3)} \rightarrow \Gc^c \rightarrow \Gc^a \rightarrow 1$, we deduce that $f \cup g \in H^2(\Gc^c)$ is precisely $d_2(\xi)$ when $\xi$ is considered as an element of $H^1(\Gc^{(2)}/\Gc^{(3)})^{\Gc^a}$.
In particular, $f \cup g$ is in the kernel of $H^2(\Gc^a) \rightarrow H^2(\Gc^c)$.

To show the other equivalence, it suffices to assume with no loss that $\Gc^c = \Gc$.
Denote by $x_1 = f$ and $x_2 = g$ and complete $x_1,x_2$ to a $\Z/\ell$-basis $(x_i)_i$ of $H^1(\Gc) = \Hom(\Gc,\Z/\ell)$.
We obtain a dual minimal generating set $(\sigma_i)_i$ for $\Gc^a$ which converges to $1$ such that $\ker x_j = \langle \sigma_i \rangle_{i \neq j}$.
Choose a corresponding free presentation $S \twoheadrightarrow \Gc$ where $S$ is the free pro-$\ell$ group on $(\sigma_i)_i$ such that $S^a \xrightarrow{\cong} \Gc^a$ is an isomorphism, and denote by $T$ the kernel of $S \rightarrow \Gc$.

Assume that $x_1 \cup x_2 \neq 0 \in H^2(\Gc)$.
Then there exists $t \in T$ such that $a_{12}(t) \neq 0$.
We can assume without loss that $a_{12}(t) = 1$.
Writing $t = [\sigma_1,\sigma_2] \cdot \rho^{-1}$, $\rho \in S^{(2)}$ with $a_{12}(\rho) = 0$, we deduce that $[\sigma_1,\sigma_2] = \rho$ when considered as elements of $\Gc$.
Thus, we deduce $(3)$.

Conversely, assume $(3)$.
Then there exists $\rho \in S^{(2)}$ such that $a_{12}(\rho) = 0$ and $[\sigma_1,\sigma_2] \cdot \rho^{-1} = t \in T$.
But then $x_1 \cup x_2 \neq 0$ since $a_{12}(t) \neq 0$.

The pro-$\ell$ case is similar.
Indeed, the proof of the equivalence of (1) and (2) is virtually the same.
To show the equivalence of (2) and (3) using a similar argument as above, we note that if $f,g \in \widehat H^1(\Gc) \smallsetminus \ell \cdot \widehat H^1(\Gc)$ are $\Z_\ell$-independent, then, for any $n \geq 0$, one has $f \cup g \neq 0$ if and only if $\ell^n \cdot (f \cup g) = (\ell^n \cdot f) \cup g \neq 0$.
This can be easily deduced from Proposition \ref{prop:central-kernel-pairig-with-H-2} since $f,g$ can be extended to a $\Z_\ell$-basis for $\widehat H^1(\Gc)$; then, in the notation of the proposition, the homomorphism $(\bullet,f \cup g)$ is non-trivial if and only if $(\bullet, \ell \cdot (f \cup g))$ is non-trivial since $\Z_\ell$ is torsion-free.
In particular, to show the equivalence of (2) and (3), we can assume that $f,g$ are actually elements of $\widehat H^1(\Gc) \smallsetminus \ell \cdot \widehat H^1(\Gc)$ so that they can be extended to a $\Z_\ell$-basis of $\widehat H^1(\Gc)$, and then proceed in the same way as the mod-$\ell$ case.
\end{proof}

\begin{defn}
\label{defn:ACL}
Let $\Gc$ be a pro-$\ell$ group.
A subgroup $A \leq \Gc^a$ is called {\bf almost commuting-liftable} (or {\bf ACL} for short), if for all $\sigma,\tau \in A$ one has $[\sigma,\tau] \in A^\beta$.

Let $A \leq \Gc^a$ be any subgroup.
We define $\Ib(A) \leq A$ as 
\[ \Ib(A) = \{ \sigma \in A \ : \ \forall \tau \in A, \ [\sigma,\tau] \in A^\beta \}. \]
Since $[\bullet,\bullet]$ is bilinear and $(\bullet)^\beta$ is linear (recall that $(\bullet)^\beta$ is trivial if $\ell = 2$), $\Ib(A)$ is a subgroup of $A$.
Observe that if $\ell = 2$, $\Ib(A)$ denotes the set of all $\sigma \in A$ such that, for all $\tau \in A$, one has $[\sigma,\tau] = 0$.
Moreover, $A = \Ib(A)$ if and only if $A$ is ACL.

A subgroup $\widehat A \leq \Gc^{ab}$ is called {\bf commuting-liftable} (or {\bf CL} for short), if for all $\sigma,\tau \in \widehat A$ one has $[\sigma,\tau] = 0$.
Let $\widehat A \leq \Gc^{ab}$ be any subgroup.
We define $\Ib(\widehat A) \leq \widehat A$ as 
\[ \Ib(\widehat A) = \{ \sigma\in \widehat A \ : \ \forall \tau \in \widehat A, \ [\sigma,\tau] = 0 \}. \]
As above, $\Ib(\widehat A)$ is a subgroup of $\widehat A$ and $\Ib(\widehat A) = \widehat A$ if and only if $\widehat A$ is CL.

Note that, in the case where $\widehat A = \langle \sigma,\tau \rangle$ has rank 2, the statement ``$\widehat A$ is CL'' is equivalent to ``$\sigma,\tau$ are commuting-liftable'' as defined in \cite{Bogomolov2007} and \cite{Pop2010}.
\end{defn}

\subsection{Galois Cohomology}
\label{sec:galois-cohomology}

Throughout this subsection, $K$ will denote a field such that $\Char K \neq \ell$ and $\mu_\ell \subset K$ (unless otherwise noted).
We will denote by $\Gc_K = \Gal(K(\ell)|K)$ the maximal pro-$\ell$ Galois group of $K$.
And we denote by $\Gc_K^a = \Gc_K/\Gc_K^{(2)}$ and $\Gc_K^c = \Gc_K/\Gc_K^{(3)}$ as defined above.

In order to avoid notational confusion, we will generally denote by:
\[ \Pi_K^a = \Gc_K^{ab} = \Gc_K/\Gc_K^{(2,\infty)}, \ \ \Pi_K^c = \Gc_K^C = \Gc_K/\Gc_K^{(3,\infty)} \]
in order to stay in line with the notation from the introduction and from \cite{Pop2010}.

Our fixed isomorphism $\mu_\ell \cong \Z_\ell$ allows us to explicitly express the Bockstein morphism $\beta : H^1(G_K,\Z/\ell) \rightarrow H^2(G_K,\Z/\ell)$ using Milnor K-theory as follows.
The cup product $H^1(G_K,\Z/\ell) \otimes \mu_\ell \xrightarrow{\cup} H^2(G_K,\mu_\ell)$ is precisely the map $\beta \cup {\bf 1}$ (cf. \cite{Efrat2011b} Proposition 2.6 and/or the proof of \cite{koch2002} Theorem 8.13).
Recall that $\omega$ denotes the fixed element of $\mu_\ell$ which corresponds to $1 \in \Z/\ell$ under our isomorphism.
This isomorphism induces isomorphisms $H^1(G_K,\Z/\ell) \cong H^1(G_K,\mu_\ell) \cong K_*^M(K)/\ell$ and $H^2(G_K,\Z/\ell) \cong H^2(G_K,\mu_{\ell}^{\otimes 2}) \cong K_2^M(K)/\ell$.
Under these induced isomorphisms, we deduce that the Bockstein morphism $H^1(G_K,\Z/\ell) \rightarrow H^2(G_K,\Z/\ell)$ corresponds to the map $K_1^M(K)/\ell \rightarrow K_2^M(K)/\ell$ defined by $x \mapsto \{x,\omega\}$.
Namely, the following diagram commutes:
\[
\xymatrix
{
K_1^M(K)/\ell \ar[r]^{\cong} \ar[d]_{x \mapsto \{x,\omega\}} & H^1(K,\mu_\ell) \ar[r]^{\cong} \ar[d]^{\text{induced}} & H^1(G_K) \ar[d]^{\beta} \\ 
K_2^M(K)/\ell \ar[r]^{\cong}  & H^1(K,\mu_\ell^{\otimes 2}) \ar[r]^{\cong} & H^2(G_K)
}
\]
where the isomorphisms on the left are canonical given by the Galois symbol, while the isomorphisms on the right are induced by our fixed isomorphism $\mu_\ell = \langle \omega \rangle \cong \Z/\ell$.
We will use this fact in the remainder of the paper without reference to this commutative diagram.

\begin{prop}
\label{prop:acl-valuation-prop}
Let $(K,v)$ be a valued field such that $\Char K \neq \ell$ and $\mu_{2\ell} \subset K$.
Denote by $D_\mu = \Gal(K'|K(\sqrt[\ell]{1+\mf_v},\sqrt[\ell]{\mu_\ell}))$.
Observe that $I_v^1 \leq D_\mu \leq D_v^1$.
Let $\sigma \in I_v^1$, $\tau' \in D_\mu$ and $\tau \in D_v^1 \smallsetminus D_\mu$.
Then:
\begin{itemize}
\item $[\tau',\sigma] = 0$.
\item $[\tau,\sigma] \in  \langle\sigma^\beta \rangle$.
\end{itemize}

Assume furthermore that $\mu_{\ell^\infty} \subset K$.
Let $\sigma \in \widehat I_v^1$ and $\tau \in \widehat D_v^1$.
Then $[\sigma,\tau] = 0$. 
\end{prop}
\begin{proof}
Pick a free presentation $S \twoheadrightarrow \Gc_K$ which induces an isomorphism $S^a \rightarrow \Gc_K^a$ and denote by $R$ the kernel of the induced map $S^c \twoheadrightarrow \Gc_K^c$.

First, denote by $T = \langle \sigma,\tau' \rangle^\perp$.
Assume without loss of generality that $\sigma$ and $\tau'$ are $\Z/\ell$-independent.
Moreover, we can complete these to a basis of $\Gc^a$, pick lifts in $\Gc$ and take $S \rightarrow \Gc$ to be the corresponding free presentation -- namely, we consider $\sigma,\tau'$ as elements of $S$ as well.
By Theorem \ref{thm:rigid-to-valuative} we deduce that $T$ is rigid; moreover, by Lemma \ref{lem:rigid-iff-wedge}, we deduce that $K_2^M(K)/T$ is a 1-dimensional quotient of $K_2^M(K)/\ell$.
Consider the perfect pairing $R \times H^2(K) \rightarrow \Z/\ell$ of Proposition \ref{prop:central-kernel-pairig-with-H-2}.
The Galois symbol $K_2^M(K)/\ell \rightarrow H^2(K)$ is an isomorphism and so the 1-dimensional quotient $K_2^M(K)/T$ corresponds to a 1-dimensional subgroup $R_T$ of $R$ via this pairing.
Denote by $\omega$ the fixed element of $\mu_\ell \subset K$ which corresponds to $1 \in \Z/\ell$ under our fixed isomorphism $\Z/\ell \cong \mu_\ell$ of $G_K$-modules.
As $\omega \in T$ and the Bockstein map $H^1(K) \rightarrow H^2(K)$ corresponds to the map $x \mapsto \{x,\omega\}$, we deduce, using the properties of the pairing $R \times H^2(K) \rightarrow \Z/\ell$ as described in Proposition \ref{prop:central-kernel-pairig-with-H-2}, that $[\tau',\sigma]$ is a generator of the 1-dimensional subgroup $R_T$.
In particular, $[\tau',\sigma] = 0$ in $\Gc_K^c$.

For the second statement, consider $T =  \langle \sigma,\tau \rangle^\perp \subset T' = \langle \sigma \rangle^\perp$. 
The rest of the argument is similar to that given above using the fact that $T' = \langle \omega, T \rangle$ and that $x \mapsto \{x,\omega\}$ corresponds to the Bockstein map in cohomology.
The pro-$\ell$ case is similar, and in fact easier since the assumption $\mu_{\ell^\infty} \subset K$ implies the triviality of the Bockstein morphism.
\end{proof}

\begin{lem}
\label{lem:acl-rigid-equiv-of-Z}
Let $K$ be a field such that $\Char K \neq \ell$ and $\mu_{2\ell} \subset K$.
Let $A \leq \Gc_K^a$ be given.
Then $(\Ib(A))^\perp = \Hb(A^\perp)$.

Assume furthermore that $\mu_{\ell^\infty} \subset K$.
Let $\widehat A \leq \Gc_K^{ab} = \Pi_K^a$ be given and assume that $\Pi_K^a/\widehat A$ is torsion-free.
Then $(\Ib(\widehat A))^\perp = \Hb(\widehat A^\perp)$.
\end{lem}
\begin{proof}
Denote by $T = A^\perp$ and $H = \Ib(A)^\perp$.
Take $(\sigma_i)_i$ a minimal generating set of $A$ and pick lifts $\tilde\sigma_i \in \Gc_K^c$ for $\sigma_i \in A \leq \Gc_K^a$; denote by $\widetilde A := \langle \tilde\sigma_i \rangle_i$ -- observe that $\widetilde A \cap (\Gc_K^c)^{(2)} = \widetilde A^{(2)}$.
Consider the commutative diagram:
\[
\xymatrix{
H^1(\Gc_K^c) \otimes H^1(\Gc_K^c) \ar[d]_{\res \otimes \res} \ar[r]^-{\cup} & H^2(\Gc_K^c) \ar[d]^{\res} \\
H^1(\widetilde A) \otimes H^1(\widetilde A)  \ar[r]^-{\cup} & H^2(\widetilde A) 
}
\]
If we identify $H^1(\Gc_K^c) = H^1(\Gc_K^a)$ with $K^\times/\ell$ and $H^1(\widetilde A) = H^1(A) = K^\times/T$ via our isomorphism $\mu_\ell \cong \Z/\ell$ and Kummer theory, we note using Lemma \ref{lem:cup-product} with $\Gc = \Gc_K$ that the top map factors via $K_2^M(K)/\ell$; thus the bottom map factors via $K_2^M(K)/T$.
Let $x \in H^1(A)$ and $y \in H^1(A)$ such that $y \notin \ker(H^1(A) \rightarrow H^1(\Ib(A))$ be given and assume that $x,y$ are $(\Z/\ell)$-independent.
Then the cup product $\inf(x \cup y) \neq 0 \in H^2(\widetilde A)$ by Lemma \ref{lem:cup-product}.
From this we deduce that $\Hb(T) \leq H$ using the alternative characterization of $\Hb(T)$ of Remark \ref{rem:ZofT}.

It follows from Proposition \ref{prop:acl-valuation-prop}, Lemma \ref{lem:cup-product} and Proposition \ref{prop:rigid-kthy} resp. \ref{prop:main-kthy} in the case where $T=\Hb(T)$ resp. $T \neq \Hb(T)$, that $H \leq \Hb(T)$.
The pro-$\ell$ case is essentially identical.
\end{proof}

\subsection{Main Results}
\label{sec:main-results-acl}

All the work has already been done and all that is left is to put everything together in the context where $K$ is a field such that $\Char K \neq \ell$ and $\mu_{2\ell} \subset K$ resp. $\mu_{\ell^\infty} \subset K$.
The theorems presented in this subsection are merely the Kummer duals of the main results from \S~\ref{sec:rigid-elements}.
They are explicitly presented here as they show how to detect valuation-theoretical data using the group theoretical data encoded in $\Gc_K^c$ resp $\Pi_K^c$.
The key in converting the main results of \S \ref{sec:rigid-elements} via Kummer theory, is Lemma \ref{lem:acl-rigid-equiv-of-Z}.

\begin{thm}
\label{thm:acl-groups}
Let $K$ be a field such that $\Char K \neq \ell$ and $\mu_{2\ell} \subset K$.
Let $Z \leq \Gc_K^a$ be given and assume that $Z = \Ib(Z)$.
Then there exists $B \leq Z$ such that $Z/B$ is cyclic (possibly trivial), $B^\perp = H$ is valuative and, denoting $v = v_H$, one has:
\begin{itemize}
\item $Z \leq D_v^1$.
\item $B \leq I_v^1$.
\end{itemize}

Assume furthermore that $\mu_{\ell^\infty} \subset K$.
Let $\widehat Z \leq \Gc_K^{ab} = \Pi_K^a$ be given; assume that $\Pi_K^a / \widehat Z$ is torsion-free and $\widehat Z = \Ib(\widehat Z)$.
Then there exists $\widehat B \leq \widehat Z$ such that $\widehat Z/\widehat B$ is cyclic (possibly trivial), $\widehat B^\perp = \widehat H$ is valuative and, denoting $v = v_{\widehat H}$, one has:
\begin{itemize}
\item $\widehat Z \leq \widehat D_v^1$.
\item $\widehat B \leq \widehat I_v^1$. 
\end{itemize}
\end{thm}
\begin{proof}
This follows from Lemma \ref{lem:acl-rigid-equiv-of-Z} and Proposition \ref{prop:rigid-kthy}.
\end{proof}

\begin{remark}
\label{rem:ell-is-2}
Assume that $\ell = 2$.
The added assumption $\mu_{4} \subset K$ in the previous theorem has two uses.
First, it ensures that $-1 \in K^{\times 2}$.
But also, it ensures that the Bockstein map $\beta : H^1(G_K) \rightarrow H^2(G_K)$ is trivial.

However, we can still deduce something even if $-1 \notin K^{\times 2}$ as follows.
Denote by $Z_\beta = \Gal(K'|K(\sqrt{-1}))$.
Let $Z \leq \Gc_K^a$ be given such that for all $\sigma,\tau \in Z$ one has $[\sigma,\tau] = 0$.
Then there exists a valuation $v$ such that $Z \leq D_v^1$ and $(Z_\beta \cap Z) / (Z_\beta \cap Z \cap I_v^1)$ is cyclic.
Indeed, arguing as in Lemma \ref{lem:acl-rigid-equiv-of-Z}, we deduce that $T = Z^\perp$ is rigid -- i.e. for all $x \in K^\times \smallsetminus T$ one has $T-xT \subset T \cup x T$.
Thus, denoting $T_\beta = T \cup - T = (Z \cap Z_\beta)^\perp$, we deduce that for all $x \notin T_\beta$ one has $T \pm xT \subset T \cup xT$; thus $T_\beta$ is rigid and $-1 \in T_\beta$.
In particular, there exists an $H$, $T_\beta \leq H \leq K^\times$, such that $H$ is valuative and, denoting $v = v_H$, one has $1+\mf_v \leq T_\beta$.

Finally, let us show that, in fact, $1+\mf_v \leq T$.
Let $x \notin H$ be given such that $x \in \mf_v$.
Then $1-x \in T_\beta$ and thus $1-x \in T$.
On the other hand, suppose that $v(x) > 0$ but $x \in H$.
Then there exists $y \notin H$ such that $0 < v(y) < v(x)$ (cf. Lemma \ref{lem:v_H}).
Thus, $(1-y) \in T$ and $v(y+x-xy) = v(y)$.
In particular, $y+x-xy = y \mod H$ and so $y+x-xy \notin H$.
Thus, $(1-y)(1-x) = 1-(y+x-xy) \in T$ so that $1-x \in T$ as required.
\end{remark}

\begin{thm}
\label{thm:main-acl}
Let $K$ be a field such that $\Char K \neq \ell$ and $\mu_{2\ell} \subset K$.
Let $Z \leq \Gc_K^a$ be given and assume that $\Ib(Z) \neq Z$.
Then $H = \Ib(Z)^\perp$ is valuative.
Denoting $v = v_H$ one has:
\begin{itemize}
\item $Z \leq D_v^1$.
\item $\Ib(Z) = Z \cap I_v^1$.
\end{itemize}

Assume furthermore that $\mu_{\ell^\infty} \subset K$.
Let $\widehat Z \leq \Gc_K^{ab} = \Pi_K^a$ be given; assume that $\Pi_K^a / \widehat Z$ is torsion-free and $\Ib(\widehat Z) \neq  \widehat Z$.
Then $\widehat H = \Ib(\widehat Z)^\perp$ is valuative.
Denoting $v = v_H$ one has:
\begin{itemize}
\item $\widehat Z \leq \widehat D_v^1$.
\item $\Ib(\widehat Z) = \widehat Z \cap \widehat I_v^1$.
\end{itemize}

\end{thm}
\begin{proof}
This follows from Lemma \ref{lem:acl-rigid-equiv-of-Z} and Proposition \ref{prop:main-kthy}.
\end{proof}

\begin{remark}
\label{rem:sub-of-acl-is-acl}
As an immediate application of Theorems \ref{thm:acl-groups}, \ref{thm:main-acl} along with Proposition \ref{prop:acl-valuation-prop}, we deduce the following non-obvious fact.
Assume that $\Char K \neq \ell$ and $\mu_{2\ell} \subset K$ as usual.
Let $A \leq \Gc_K^a$ be given, assume that $A \neq \Ib(A)$ and take $\sigma \in \Ib(A)$, $\tau \in A$.
Then $[\sigma,\tau] \in \langle \sigma^\beta \rangle$.
Similarly, we deduce that any subgroup of an ACL subgroup $A \leq \Gc_K^a$ is also ACL.
\end{remark}

\begin{thm}
\label{thm:acl-groups-with-char}
Let $K$ be a field such that $\Char K \neq \ell$ and $\mu_{2\ell} \subset K$.
Let $Z \leq \Gc_K^a$ be given and denote by $L = (K')^Z$.
Suppose there exists a subgroup $Z_0 \leq \Gc_L^a$ which maps surjectively onto $Z$ via the canonical map $\Gc_L^a \rightarrow \Gc_K^a$ such that $\Ib(Z_0) = Z_0$.
Then there exists $B \leq Z$ such that $Z/B$ is cyclic (possibly trivial), $B^\perp = H$ is valuative and, denoting $v = v_H$, one has:
\begin{itemize}
\item $Z \leq D_v^1$.
\item $B \leq I_v^1$.
\item $\Char k(v) \neq \ell$ and thus $D_v^1 = D_v$ and $I_v^1 = I_v$.
\end{itemize}

Assume furthermore that $\mu_{\ell^\infty} \subset K$.
Let $\widehat Z \leq \Gc_K^{ab} = \Pi_K^a$ be given such that $\Pi_K^a/\widehat Z$ is torsion-free and denote by $L = K\left(\sqrt[\ell^\infty]{K^\times \cap \widehat Z^\perp}\right) = (K^{ab})^{\widehat Z}$.
Suppose there exists a subgroup $\widehat Z_0 \leq \Gc_L^{ab} = \Pi_L^a$ such that $\Pi_L^a/\widehat Z_0$ is torsion-free and which maps surjecively onto $\widehat Z$ via the canonical map $\Pi_L^a \rightarrow \Pi_K^a$, such that $\Ib(\widehat Z_0) = \widehat Z_0$.
Then there exists $\widehat B \leq \widehat Z$ such that $\widehat Z/\widehat B$ is cyclic (possibly trivial), $\widehat B^\perp = \widehat H$ is valuative and, denoting $v = v_{\widehat H}$, one has:
\begin{itemize}
\item $\widehat Z \leq \widehat D_v^1$.
\item $\widehat B \leq \widehat I_v^1$.
\item $\Char k(v) \neq \ell$ and thus $\widehat D_v^1 = \widehat D_v$ and $\widehat I_v^1 = \widehat I_v$.
\end{itemize}
\end{thm}
\begin{proof}
This follows from Lemma \ref{lem:acl-rigid-equiv-of-Z} and Proposition \ref{prop:rigid-kthy-with-char}.
\end{proof}

\begin{thm}
\label{thm:main-acl-with-char}
Let $K$ be a field such that $\Char K \neq \ell$ and $\mu_{2\ell} \subset K$.
Let $Z \leq \Gc_K^a$ be given and assume that $\Ib(Z) \neq Z$.
Denote by $L = (K')^Z$ and assume furthermore that there exists a subgroup $Z_0 \leq \Gc_L^a$ such that $\Ib(Z_0) \leq Z_0$ map surjectively onto $\Ib(Z) \leq Z$ via the canonical map $\Gc_L^a \rightarrow \Gc_K^a$.
Then $H = \Ib(Z)^\perp$ is valuative.
Denoting $v = v_H$ one has:
\begin{itemize}
\item $Z \leq D_v^1$.
\item $\Ib(Z) = Z \cap I_v^1$.
\item $\Char k(v) \neq \ell$ and thus $D_v^1 = D_v$ and $I_v^1 = I_v$.
\end{itemize}

Assume furthermore that $\mu_{\ell^\infty} \subset K$.
Let $\widehat Z \leq \Gc_K^{ab} = \Pi_K^a$ be given such that $\Pi_K^a/\widehat Z$ is torsion-free and assume that $\Ib(\widehat Z) \neq \widehat Z$.
Denote by $L = K\left(\sqrt[\ell^\infty]{K^\times \cap \widehat Z^\perp}\right) = (K^{ab})^{\widehat Z}$ and assume furthermore that there exists a subgroup $\widehat Z_0 \leq \Gc_L^{ab} = \Pi_L^a$ such that $\Pi_L^a/\widehat Z_0$ is torsion-free and $\Ib(\widehat Z_0) \leq \widehat Z_0$ map surjectively onto $\Ib(\widehat Z) \leq \widehat Z$ via the canonical map $\Pi_L^a \rightarrow \Pi_K^a$.
Then $\widehat H = \Ib(\widehat Z)^\perp$ is valuative.
Denoting $v = v_H$ one has:
\begin{itemize}
\item $\widehat Z \leq \widehat D_v^1$.
\item $\Ib(\widehat Z) = \widehat Z \cap \widehat I_v^1$.
\item $\Char k(v) \neq \ell$ and thus $\widehat D_v^1 = \widehat D_v$ and $\widehat I_v^1 = \widehat I_v$.
\end{itemize}

\end{thm}
\begin{proof}
This follows from Lemma \ref{lem:acl-rigid-equiv-of-Z} and Proposition \ref{prop:main-kthy-with-char}.
\end{proof}

To finish this subsection, we state the Kummer dual of Theorem \ref{thm:main-kthy-compatibility}.
In particular, this gives a group-theoretical recipe to recover {\bf precisely} the subgroups $I_v^1 \leq D_v^1$ resp. $\widehat I_v^1 \leq \widehat D_v^1$ for valuations $v \in \Vc_K$ resp. $v \in \widehat \Vc_K$ using only the group-theoretical data encoded in $\Gc_K^c$ resp. $\Pi_K^c$.

\begin{defn}
\label{defn:Ac}
Let $K$ be a field such that $\Char K \neq \ell$ and $\mu_{2\ell} \subset K$.
We denote by $\Ac_K$ the collection of subgroups $Z \leq \Gc_K^a$ such that
\begin{itemize}
\item $Z \neq \Ib(Z)$.
\item $\Ib(Z) \leq Z$ is maximal; i.e. whenever $Z \leq Z' \leq \Gc_K^a$ and $Z \neq Z'$, one has $\Ib(Z) \not\leq \Ib(Z')$.
In particular, this implies that $\Ib(Z) \neq 1$ or $Z = \Gc_K^a$.
\end{itemize}

On the other hand, assume that $\mu_{\ell^\infty} \subset K$.
Denote by $\widehat \Ac_K$ the collection of all subgroups $\widehat Z \leq \Gc_K^{ab} = \Pi_K^a$ such that:
\begin{itemize}
\item $\widehat Z \neq \Ib(\widehat Z)$ and $\Pi_K^a/\widehat Z$ is torsion-free.
\item $\Ib(\widehat Z) \leq \widehat Z$ is maximal; i.e. whenever $\widehat Z \leq \widehat Z' \leq \Pi_K^a$, $\Pi_K^a/\widehat Z'$ is torsion-free and $\widehat Z \neq \widehat Z'$, one has $\Ib(\widehat Z) \not\leq \Ib(\widehat Z')$.
\end{itemize}
\end{defn}

\begin{remark}
\label{rem:centralizer}
Assume that $\mu_{2\ell} \subset K$.
Using Lemma \ref{lem:acl-rigid-equiv-of-Z}, we deduce that the map $Z \mapsto Z^\perp$ induces a bijection between $\Ac_K$ and $\Tc_K$ ($Z^\perp$ denotes the Kummer orthogonal of $Z \leq \Gc_K^a$); similarly, the map $\widehat Z \mapsto \widehat Z^\perp$ induces a bijection between $\widehat \Ac_K$ and $\widehat \Tc_K$ if further $\mu_{\ell^\infty} \subset K$.

On the other hand, we can explicitly describe the elements of $\Ac_K$ and $\widehat \Ac_K$ in a simpler way.
First let us discuss the mod-$\ell$ case -- assume that $\mu_{2\ell} \subset K$.
Let $A \leq \Gc_K^a$ be any given ACL subgroup, and denote by 
\[ C(A) = \{\sigma\in \Gc_K^a \ : \ \forall \tau \in A, \ [\sigma,\tau] \in A^\beta \}. \]
If $\Ib(C(A)) \neq C(A)$, then $C(A) \in \Ac_K$.
Indeed, one has $A \leq \Ib(C(A)) \leq C(A)$; let $C(A) \leq Z'$ be given such that $C(A) \neq Z'$ and assume that $\Ib(C(A)) \leq \Ib(Z')$.
Then, in particular, $A \leq \Ib(Z')$.
But this implies that $Z' \leq C(A)$ by Theorem \ref{thm:main-acl} and Proposition \ref{prop:acl-valuation-prop} -- note that $Z' \neq \Ib(Z')$ and recall that for all $\sigma \in \Ib(Z')$ and $\tau \in Z'$ one has $[\sigma,\tau] \in \langle \sigma^\beta \rangle$.
Conversely, let $Z \in \Ac_K$ be given.
Then the maximality condition on $\Ib(Z) \leq Z$ implies that $Z = C(\Ib(Z))$ (here one uses Theorem \ref{thm:main-acl} and Proposition \ref{prop:acl-valuation-prop} again).
In particular, we deduce that $Z \in \Ac_K$ if and only if $Z = C(\Ib(Z))$ and $Z \neq \Ib(Z)$ -- i.e. $Z$ is the ``almost-commuting-liftable-centralizer'' of its ``almost-commuting-liftable-center.''

On the other hand, assume that $\mu_{\ell^\infty} \subset K$. For $\widehat A \leq \Pi_K^a$ such that $\Pi_K^a/\widehat A$ is torsion-free and $\widehat A$ is CL, we denote by
\[ C(\widehat A) = \{ \sigma\in \Pi_K^a \ : \ \forall \tau \in \widehat A, \ [\sigma,\tau] = 0 \}. \]
A similar argument as above shows that $C(\widehat A) \in \widehat \Ac_K$ if $\Ib(C(\widehat A)) \neq C(\widehat A)$.
Moreover, $\widehat Z \in \widehat \Ac_K$ if and only if $C(\Ib(\widehat Z)) = \widehat Z$ and $\widehat Z \neq \Ib(\widehat Z)$ -- i.e. $\widehat Z$ is the ``commuting-liftable-centralizer'' of its ``commuting-liftable-center.'' 
\end{remark}

\begin{thm}
\label{thm:main-ACL-compatibility}
Let $K$ be a field such that $\Char K \neq \ell$ and $\mu_{2\ell} \subset K$.
Then the map $\Vc_K \rightarrow \Ac_K$ defined by $v \mapsto D_v^1$ is a bijection which is compatible with coarsening of $v$ as in Theorem \ref{thm:main-kthy-compatibility}.
Moreover, for $v \in \Vc_K$ one has $\Ib(D_v^1) = I_v^1$.

Assume furthermore that $\mu_{\ell^\infty} \subset K$.
Then the map $\widehat \Vc_K \rightarrow \widehat \Ac_K$ defined by $v \mapsto \widehat D_v^1$ is a bijection which is compatible with coarsening of $v$ as in Theorem \ref{thm:main-kthy-compatibility}.
Moreover, for $v \in \widehat \Vc_K$ one has $\Ib(\widehat D_v^1) = \widehat I_v^1$.
\end{thm}
\begin{proof}
This follows immediately from Theorem \ref{thm:main-kthy-compatibility} and Lemma \ref{lem:acl-rigid-equiv-of-Z}.
\end{proof}

\begin{remark}
\label{rem:A-equals-Z-of-A}
Arguing as in Remark \ref{rem:T-euqals-Z-of-T}, we deduce the following if $\mu_{2\ell} \subset K$.
Suppose $Z \leq \Gc_K^a$ is an ACL subgroup which satisfies the maximality conditions of Definition \ref{defn:Ac}.
Then there exists a valuation $v$ of $K$ such that $\Gamma_v$ contains no non-trivial $\ell$-divisible convex subgroups, $k(v)^\times/\ell$ is cyclic and $Z = D_v^1$.

Similarly, if $\mu_{\ell^\infty} \subset K$ and $\widehat Z \leq \Gc_K^{ab} = \Pi_K^a$ is a CL subgroup which satisfies the maximality condition, then there exists a valuation $v$ of $K$ such that $\Gamma_v$ contains no non-trivial $\ell$-divisible convex subgroups, $\widehat {k(v)}$ is cyclic and $\widehat Z = \widehat D_v^1$.
\end{remark}

\section{Applications}
\label{sec:applications}

\subsection{Detecting Inertia/Decomposition Groups}
\label{sec:detect-inert-groups}

Throughout this subsection we let $K$ be a field such that $\Char K \neq \ell$ and $\mu_{\ell} \subset K$.
If $\Char K > 0$, then $\Char k(v) = \Char K$ for all valuations $v$ of $K$.
Thus, combining Proposition \ref{prop:decomp-thy-prop} and Theorem \ref{thm:main-ACL-compatibility}, we immediately deduce the following corollary:
\begin{cor}
\label{cor:pos-char-group-thy-recipe}
Suppose $K$ is a field such that $\ell \neq \Char K > 0$ and $\mu_{2\ell} \subset K$.
Then the elements of $Z \in \Ac_K$ are precisely the decomposition groups of valuations $v \in \Vc_K$, and the bijection $\Vc_K \rightarrow \Ac_K$ of Theorem \ref{thm:main-ACL-compatibility} is precisely $v \mapsto D_v$.
Moreover, $\Ib(D_v) = I_v$ for $v \in \Vc_K$.

Assume furthermore $\mu_{\ell^\infty} \subset K$.
Then the elements $\widehat Z \in \widehat \Ac_K$ are precisely the decomposition groups of valuations $v \in \widehat \Vc_K$, and the bijection $\widehat \Vc_K \rightarrow \widehat \Ac_K$ of Theorem \ref{thm:main-ACL-compatibility} is precisely $v \mapsto \widehat D_v$.
Moreover, $\Ib(\widehat D_v) = \widehat I_v$ for $v \in \widehat \Vc_K$.
\end{cor}
This can be seen as a group theoretical recipe which detects the inertia/decomposition subgroups $I_v \leq D_v$ inside $\Gc_K^a$ using only the group theoretical data encoded in $\Gc_K^c$, in the case where $\ell \neq \Char K > 0$ and $\mu_{2\ell} \subset K$.
If $\mu_{\ell^\infty} \subset K$ and $\ell \neq \Char K > 0$ this gives a group theoretical recipe which detects the inertia/decomposition subgroups $\widehat I_v \leq \widehat D_v$ of $\Pi_K^a$ using the group theoretical data encoded in $\Pi_K^c$.

\begin{example}
\label{example:char-0-ff}
Consider $X = \Abb^1_{\Z[\mu_{\ell^2}]}$, a scheme over $\Z[\mu_{\ell^2}]$.
Denote by $K = \Q(\mu_{\ell^2},x)$ the function field of $X$ and consider the valuation $v$ associated to the vertical divisor of $X$ over $\ell \in \Spec\Z[\mu_{\ell^2}]$.
Arguing in a similar way to Example \ref{example:div-vals}, we have $v \in \Vc_K$.

On the other hand, $\Gal(k(v)(\ell)|k(v)) = \Gc_{k(v)}(\ell)$ is a non-trivial pro-$\ell$ group.
In particular, $D_v \neq I_v$ as subgroups of $\Gc_K^a$ while $I_v^1 \leq D_v^1 \leq I_v$ by Remark \ref{rem:charkv-is-ell}.
A similar situation holds in the pro-$\ell$ case for the scheme $X = \Abb^1_{\Z[\mu_{\ell^\infty}]}$ considered as a scheme over $\Z[\mu_{\ell^\infty}]$.
\end{example}

In particular, we deduce that an analogous statement as Corollary \ref{cor:pos-char-group-thy-recipe} is false for general fields of characteristic 0.
On the other hand, we can detect precisely which $v \in \Vc_K$ resp. $\widehat \Vc_K$ have $\Char k(v) = \ell$ using Theorem \ref{thm:main-acl-with-char}.
First, we need a lemma which can be seen as a converse to Theorems \ref{thm:acl-groups-with-char} and \ref{thm:main-acl-with-char}.

\begin{lem}
\label{lem:liftable-acl-char}
Let $K$ be a field such that $\Char K \neq \ell$ and $\mu_{2\ell} \subset K$.
Let $Z \leq \Gc_K^a$ be given and take $L$ any field such that $L \cap K' = (K')^Z$.
Assume that $\Ib(Z) \neq Z$ and denote by $H = \Ib(Z)^\perp$.
Recall that $H$ is valuative, denote by $v = v_H$ and assume that $\Char k(v) \neq \ell$.
Then there exists $Z_0 \leq \Gc_L^a$ such that $\Ib(Z_0) \leq Z_0$ map surjectively onto $\Ib(Z) \leq Z$ via the canonical map $\Gc_L^a \rightarrow \Gc_K^a$.

Assume furthermore that $\mu_{\ell^\infty} \subset K$.
Let $\widehat Z \leq \Pi_K^a$ be given such that $\Pi_K^a / \widehat Z$ is torsion-free, and take $L$ any field such that $L \cap K^{ab} = (K^{ab})^{\widehat Z}$.
Assume that $\Ib(\widehat Z) \neq \widehat Z$ and denote by $\widehat H = \Ib(\widehat Z)^\perp$.
Recall that $\widehat H$ is valuative, denote by $v = v_{\widehat H}$ and assume that $\Char k(v) \neq \ell$.
Then there exists $\widehat Z_0 \leq \Pi_L^a$ such that $\Pi_L^a / \widehat Z_0$ is torsion-free and $\Ib(\widehat Z_0) \leq \widehat Z_0$ map surjectively onto $\Ib(\widehat Z) \leq \widehat Z$ via the canonical map $\Pi_L^a \rightarrow \Pi_K^a$.
\end{lem}
\begin{proof}
Below we show the mod-$\ell$ case as the pro-$\ell$ case is essentially the same.
This follows from the discussion of \S~\ref{sec:hilb-decomp-theory} as follows.
Take a prolongation $w|v$ to $L|K$.
The canonical map $\Gc_L^a \rightarrow \Gc_K^a$ induces corresponding maps $D_w \rightarrow D_v$ and $I_w \rightarrow I_v$.
The canonical map $D_w \rightarrow D_v$ has image $Z$ and the canonical map $I_w \rightarrow I_v$ has image $\Ib(Z)$ -- indeed, the assumption $\Char k(v) \neq \ell$ implies that $Z \leq D_v$ and $\Ib(Z) = Z \cap I_v$.
On the other hand, $\Ib(D_w) = I_w$, as required.
\end{proof}

\begin{remark}
Using an argument similar to that in Lemma \ref{lem:liftable-acl-char} we deduce the following:
Let $K$ be a field such that $\Char K \neq \ell$ and $\mu_{2\ell} \subset K$.
Take $Z \leq \Gc_K^a$ such that $\Ib(Z) = Z$.
By Theorem \ref{thm:acl-groups}, there exists $Z' \leq Z$ such that $Z/Z'$ is cyclic, $H = (Z')^\perp$ is valuative and $1+\mf_{v_H} \subset Z^\perp$.
Denote by $L = K(\sqrt[\ell]{Z^\perp}) = (K')^{Z}$ and assume furthermore that there exists an $Z'$ as above such that $\Char k(v_H) \neq \ell$.
Then there exists $B \leq \Gc_L^a$ such that $B = \Ib(B)$ and $B$ maps surjectively onto $Z$ via the canonical map $\Gc_L^a \rightarrow \Gc_K^a$.

Assume furthermore that $\mu_{\ell^\infty} \subset K$.
Take $\widehat Z \leq \Gc_K^{ab} = \Pi_K^a$ such that $\Ib(\widehat Z) = \widehat Z$.
By Theorem \ref{thm:acl-groups}, there exists $\widehat Z' \leq \widehat Z$ such that $\widehat Z/\widehat Z'$ is cyclic, $\widehat H  = (\widehat Z')^\perp$ is valuative and $1+\mf_{v_{\widehat H}} \subset \widehat Z^\perp$.
Denote by $L = K\left(\sqrt[\ell^\infty]{K^\times \cap \widehat Z^\perp}\right) = (K^{ab})^{\widehat Z}$ and assume furthermore that there exists $\widehat Z'$ as above such that $\Char k(v_{\widehat H}) \neq \ell$.
Then there exists $\widehat B \leq \Gc_L^{ab} = \Pi_L^a$ such that $\Pi_L^a/\widehat B$ is torsion-free, $\widehat B = \Ib(\widehat B)$ and $\widehat B$ maps surjectively onto $\widehat Z$ via the canonical map $\Pi_L^a \rightarrow \Pi_K^a$.
\end{remark}

For $v \in \Vc_K$, denote by $L_v = K(\sqrt[\ell]{1+\mf_v}) = (K')^{D_v^1}$; recall using Theorem \ref{thm:main-ACL-compatibility} that
\begin{itemize}
\item $D_v^1 \in \Ac_K$ and the map $v \mapsto D_v^1$ is a bijection between $\Vc_K$ and $\Ac_K$.
\item $\Ib(D_v^1) = I_v^1$.
\end{itemize}
In the pro-$\ell$ context for $v \in \widehat \Vc_K$ we denote by $L_{\widehat v} = K(\sqrt[\ell^\infty]{1+\mf_v}) = (K^{ab})^{\widehat D_v^1}$; recall using Theorem \ref{thm:main-ACL-compatibility} that
\begin{itemize}
\item $\widehat D_v^1 \in \widehat \Ac_K$ and the map $v \mapsto \widehat D_v^1$ is a bijection between $\widehat \Vc_K$ and $\widehat \Ac_K$.
\item $\Ib(\widehat D_v^1) = \widehat I_v^1$.
\end{itemize}

\begin{cor}
\label{cor:char-0-group-thy-recipe}
Suppose $K$ is a field such that $\Char K \neq \ell$ and $\mu_{2\ell} \subset K$.
Let $Z \in \Ac_K$ be given and take $v \in \Vc_K$ such that $Z = D_v^1$.
Then $\Char k(v) \neq \ell$ if and only if there exists $B \leq \Gc_{L_v}^a$ such that $\Ib(B) \leq B$ map surjectively onto $\Ib(Z) \leq Z$ via the canonical map $\Gc_{L_v}^a \rightarrow \Gc_K^a$.
Moreover, if these equivalent conditions hold then $Z = D_v$ and $\Ib(Z) = I_v$.

Let $\widehat Z \in \widehat\Ac_K$ be given and take $v \in \widehat\Vc_K$ such that $\widehat Z = \widehat D_v^1$.
Then $\Char k(v) \neq \ell$ if and only if there exists $\widehat B \leq \Gc_{L_{\widehat v}}^{ab} = \Pi_{L_{\widehat v}}^a$ such that $\Pi_{L_{\widehat v}}^a/\widehat B$ is torsion-free and $\Ib(\widehat B) \leq \widehat B$ map surjectively onto $\Ib(\widehat Z) \leq \widehat Z$ via the canonical map $\Pi_{L_{\widehat v}}^a \rightarrow \Pi_K^a$.
Moreover, if these equivalent conditions hold then $\widehat Z = \widehat D_v$ and $\Ib(\widehat Z) = \widehat I_v$.
\end{cor}
\begin{proof}
Theorem \ref{thm:main-acl-with-char} and Lemma \ref{lem:liftable-acl-char} imply the equivalence.
The last statement then follows from Proposition \ref{prop:decomp-thy-prop}.
\end{proof}

This corollary can be seen as a group theoretical recipe which detects the decomposition/inertia subgroups in $\Gc_K^a$ resp $\Pi_K^a$ of valuations $v \in \Vc_K$ resp $\widehat \Vc_K$ such that $\Char k(v) \neq \ell$ using group theoretical data encoded in a Galois group which is larger than $\Gc_K^c$ resp. $\Pi_K^c$, which we denote by $\Gc_K^m$ resp. $\Gc_K^M$ as defined below.

First assume that $\mu_\ell \subset K$.
We denote by $K^c$ the Galois extension of $K$ such that $\Gal(K^c|K) = \Gc_K^c$.
Denote by $K^m$ the copositum of fields of the form $L^c$ for fields $K|L|K'$ and $\Gc_K^m := \Gal(K^m|K)$.

Similarly, if $\mu_{\ell^\infty} \subset K$, we denote by $K^C$ the Galois extension of $K$ such that $\Gal(K^C|K) = \Pi_K^c$.
Denote by $K^M$ the compositum of all fields of the form $L^C$ for fields $K|L|K^{ab}$ and $\Gc_K^M := \Gal(K^M|K)$.
Then Corollary \ref{cor:char-0-group-thy-recipe} provides a group-theoretical recipe to detect precisely the decomposition/inertia subgroups in $\Gc_K^a$ resp. $\Pi_K^a$ of valuations $v \in \Vc_K$ resp. $v \in \widehat \Vc_K$ such that $\Char k(v) \neq 0$ from the group-theoretical structure of $\Gc_K^m$ resp. $\Gc_K^M$.

\subsection{Maximal pro-$\ell$ Galois Groups of Fields}
\label{sec:maximal-pro-ell}

To conclude this paper, we show an interesting application to the structure of the maximal pro-$\ell$ Galois group of fields which is motivated by the example and corollaries of the previous subsection.
Let $K$ be a field such that $\Char K \neq \ell$ and $\mu_{2\ell} \subset K$.
To simplify the statement of the proposition that follows, we will denote by $\Ac_K^0$ the set consisting of all $Z \in \Ac_K$ which satisfy the equivalent conditions of Corollary \ref{cor:char-0-group-thy-recipe}.
In particular, the map $v \mapsto D_v^1$ is a bijection between the set of all $v \in \Vc_K$ such that $\Char k(v) \neq \ell$ and $\Ac_K^0$. 
Of course, one can formulate a similar statement in the pro-$\ell$ situation, which we omit, using $\Pi_K^c$ and $\Gc_K^M$; however the statement of the corollary below is in any case stronger than the analogous pro-$\ell$ one.

\begin{cor}
\label{cor:application-char}
Let $K$ be a field such that $\Char K \neq \ell$ and $\mu_{2\ell} \subset K$.
Then $K$ has a valuation $v \in \Vc_K$ such that $\Char k(v) = \ell$ if and only if there exists a subgroup $Z \in \Ac_K \smallsetminus \Ac_K^0$.
In particular, if $K$ satisfies these equivalent conditions, there does not exist any field $F$ of positive characteristic such that $\mu_{2\ell} \subset F$ and $\Gc_K^m \cong \Gc_F^m$ (an abstract isomorphism of pro-$\ell$ groups).
\end{cor}
\begin{proof}
This follows from Corollaries \ref{cor:pos-char-group-thy-recipe}, \ref{cor:char-0-group-thy-recipe} and the discussion of \S~\ref{sec:detect-inert-groups}.
\end{proof}

\begin{remark}
\label{rem:gcm-remark}
Since the assignment $\Gc_K \mapsto \Gc_K^m$ is functorial, we deduce that $\Gc_K^m \not\cong \Gc_F^m$ implies $\Gc_K \not\cong \Gc_F$.
In particular, the statement of the corollary above still holds replacing $\Gc_K^m$ resp. $\Gc_F^m$ with $\Gc_K$ resp. $\Gc_F$.
\end{remark}

\begin{example}
\label{ex:pro-ell-example}
Let $K = \Q(\mu_{\ell^2}, t)$.
Denote by $v \in \Vc_K$ the valuation associated to the vertical divisor of $\Abb^1_{\Z[\mu_{\ell^2}]}$ over $\ell \in \Spec \Z[\mu_{\ell^2}]$.
Then $D_v^1 \in \Ac_K \smallsetminus \Ac_K^0$ since $\Char k(v) = \ell$.
We deduce that there does not exist any field $F$ of positive characteristic such that $\Gc_F \cong \Gc_K$.
\end{example}

To conclude this paper, we present yet another application to the structure of the maximal pro-$\ell$ Galois group, in the case where $\mu_{\ell^\infty} \subset K$, which describes commuting elements in $\Gc_K$ versus commuting elements of its small quotients.

\begin{cor}
\label{cor:pro-ell-comm-pairs}
Let $K$ be a field such that $\Char K \neq \ell$ and $\mu_{\ell^\infty} \subset K$.
Let $\sigma,\tau \in \Gc_K^a$ be given.
Assume first that $\Char K > 0$.
Then the following are equivalent:
\begin{itemize}
\item There exist lifts $\tilde\sigma,\tilde\tau \in \Gc_K^c$ of $\sigma,\tau$ such that $\tilde\sigma\tilde\tau = \tilde\tau\tilde\sigma$ (namely $[\sigma,\tau] = 0$).
\item There exist lifts $\tilde\sigma,\tilde\tau \in \Gc_K$ of $\sigma,\tau$ such that $\tilde\sigma\tilde\tau = \tilde\tau\tilde\sigma$.
\end{itemize}
Assume, on the other hand, that $\Char K = 0$.
Then the following are equivalent:
\begin{itemize}
\item There exist lifts $\tilde\sigma,\tilde\tau \in \Gc_K^m$ of $\sigma,\tau$ such that $\tilde\sigma\tilde\tau = \tilde\tau\tilde\sigma$.
\item There exist lifts $\tilde\sigma,\tilde\tau \in \Gc_K$ of $\sigma,\tau$ such that $\tilde\sigma\tilde\tau = \tilde\tau\tilde\sigma$.
\end{itemize}
\end{cor}
\begin{proof}
This follows from the Theorems in \S\ref{sec:main-results-acl}, along with Lemma \ref{lem:liftable-acl-char} and the discussion of \S \ref{sec:hilb-decomp-theory}.
\end{proof}

\bibliographystyle{amsalpha}
\bibliography{refs} %make sure reference location is correct!!!

\end{document}